\documentclass{article}

\usepackage{microtype}
\usepackage{graphicx}
\usepackage{subfigure}
\usepackage{booktabs} 

\usepackage{hyperref}

\usepackage[accepted]{icml2019}

\icmltitlerunning{A Stochastic Trust Region Method for Non-convex Minimization}

\usepackage[utf8]{inputenc} 
\usepackage[T1]{fontenc}    
\usepackage{hyperref}       
\usepackage{url}            
\usepackage{booktabs}       
\usepackage{amsfonts}       
\usepackage{nicefrac}       
\usepackage{microtype}      
\usepackage{enumerate}
\usepackage{euscript}

\usepackage{bm}
\usepackage{color}
\usepackage{amsthm}
\usepackage{tabularx}
\usepackage{footnote}
\usepackage{threeparttable}
\usepackage{tablefootnote}
\makesavenoteenv{tabular}
\makesavenoteenv{table}
\usepackage{caption}
\usepackage{fancyhdr}

\usepackage{lettrine}

\usepackage{hyperref}

\usepackage{amsthm}
\newtheorem*{example*}{Example}
\usepackage{amsmath}
\usepackage{amssymb}

\usepackage{epsfig}
\usepackage{epstopdf}

\usepackage{graphicx}
\usepackage{subfigure}

\usepackage{algorithm}
\usepackage{algorithmic}

\usepackage{xcolor}

\usepackage{comment}

\input{required.tex}

\begin{document}

\twocolumn[
\icmltitle{A Stochastic Trust Region Method for Non-convex Minimization}



\icmlsetsymbol{equal}{*}

\begin{icmlauthorlist}
	\icmlauthor{Zebang Shen}{equal,zju}
	\icmlauthor{Pan Zhou}{equal,nus}
	\icmlauthor{Cong Fang}{pku}
	\icmlauthor{Alejandro Ribeiro}{upenn}
\end{icmlauthorlist}

\icmlaffiliation{zju}{Zhejiang University}
\icmlaffiliation{nus}{National University of Singapore}
\icmlaffiliation{pku}{Peking University}
\icmlaffiliation{upenn}{University of Pennsylvania}


\icmlkeywords{Machine Learning, ICML}

\vskip 0.3in
]



\printAffiliationsAndNotice{\icmlEqualContribution} 

\begin{abstract}
We target the problem of finding a local minimum in non-convex finite-sum minimization. Towards this goal, we first prove that the trust region method with inexact gradient and Hessian estimation can achieve a convergence rate of order $\OM({1}/{k^{2/3}})$ as long as those differential estimations are sufficiently accurate.
	Combining such result with a novel Hessian estimator, we propose the sample-efficient stochastic trust region (STR) algorithm which finds an $(\epsilon, \sqrt{\epsilon})$-approximate local minimum within $\OM({\sqrt{n}}/{\epsilon^{1.5}})$ stochastic Hessian oracle queries.
	This improves state-of-the-art result by $\OM(n^{1/6})$.
	Experiments verify theoretical conclusions and the efficiency of STR.
\end{abstract}
\section{Introduction}
We consider the following finite-sum minimization problem
\begin{equation}
\min_{\xB\in\RBB^d} F(\xB) = \frac{1}{n}\sum_{i=1}^{n} f_i(\xB),
\label{eqn_stochastic_problem}
\end{equation}
where each (non-convex) component function $f_i:\RBB^d \rightarrow \RBB$ is assumed to have $L_1$-Lipschitz continuous gradient and $L_2$-Lipschitz continuous Hessian.
Since first-order stationary points could be saddle points and thus lead to inferior generalization performance~\cite{dauphin2014identifying}, in this work we are particularly interested in computing $(\epsilon, \sqrt{\epsilon})$-approximate second-order stationary points, $\epsilon$-SOSP:
\begin{equation}
\|\nabla F(\xB_\epsilon) \|\leq \epsilon \quad \text{and}\quad  \nabla^2 F(\xB_\epsilon) \succcurlyeq -\sqrt{L_2\epsilon} \IB. \label{eqn_epsilon_approximate_second_order_stationary_point}
\end{equation}
To find the local minimum in problem~\eqref{eqn_stochastic_problem}, the cubic regularization approach~\cite{nesterov2006cubic} and the trust region algorithm~\cite{conn2000trust,curtis2017trust} are two classical methods.
Specifically, cubic regularization forms a cubic surrogate function for the objective $F(\xB)$ by adding a third-order regularization term to the second-order Taylor expansion, and minimizes it iteratively.
Such a method is proved to achieve an $\OM({1}/{k^{2/3}})$ global convergence rate and thus needs $\OM({n}/{\epsilon^{1.5}})$ stochastic first- and second-order oracle queries, namely the evaluation number of stochastic gradient and Hessian, to achieve a point that satisfies (\ref{eqn_epsilon_approximate_second_order_stationary_point}).
On the other hand, trust region algorithms estimate the objective with its second-order Taylor expansion but minimize it only within a local region.
Recently, \citet{curtis2017trust} propose a trust region variant to achieve the same convergence rate as the cubic regularization approach. But both methods require computing full gradients and Hessians of the objective $F(\xB)$ and thus suffer from high computational cost in  large-scale problems.


To avoid coslty exact differential evaluations, many works explore the finite-sum structure of problem~\eqref{eqn_stochastic_problem} and develop stochastic cubic regularization approaches.
Both \citet{pmlr-v70-kohler17a} and \citet{xu2017newton} propose to directly subsample the gradient and Hessian in the cubic surrogate function, and achieve  $\OM({1}/{\epsilon^{3.5}})$ and $\OM({1}/{\epsilon^{2.5}})$ stochastic first- and second-order oracle complexities respectively.
By plugging a stochastic variance reduced estimator \cite{johnson2013accelerating} and the Hessian tracking technique \cite{pmlr-v84-gower18a} into the gradient and Hessian estimation, the approach in \citet{pmlr-v80-zhou18d}
improves both the stochastic first- and second-order oracle complexities to $\OM({n^{0.8}}/{\epsilon^{1.5}})$. Recently, \cite{zhang2018adaptive,zhou2018sample} develop more efficient stochastic cubic regularization variants, which further reduce the stochastic second-order oracle complexities to $\OM({n^{2/3}}/{\epsilon^{1.5}})$ at the cost of increasing the stochastic first-order oracle complexity to $\OM({n^{2/3}}/{\epsilon^{2.5}})$.

\noindent{\bf Contributions:} In this paper we propose and exploit a formulation in which we make explicit control of the step size in the trust region method. This idea is leveraged to develop two efficient stochastic trust region (STR) approaches.
We tailor our methods to achieve state-of-the-art oracle complexities under the following two measurements: (i) the stochastic second-order oracle complexity is prioritized; (ii) the stochastic first- and second-order oracle complexities are treated equally. Specifically, in Setting (i), our method STR$_1$ employs a newly proposed estimator to approximate the Hessian and adopts the estimator in \cite{fang2018spider} for gradient approximation.
Our novel Hessian estimator maintains a high accuracy second-order differential approximation with lower amortized oracle complexity.
In this way, STR$_1$ achieves $\OM(\min\{{1}/{\epsilon^2}, {\sqrt{n}}/{\epsilon^{1.5}}\})$ stochastic second-order oracle complexity. This is lower than existing results for solving problem~\eqref{eqn_stochastic_problem}.
In Setting (ii), our method STR$_2$ substitutes the gradient estimator in STR$_1$ with one that integrates stochastic gradient and Hessian together to maintain an accurate gradient approximation. As a result, STR$_2$ achieves convergence in $\OM({n^{3/4}}/{\epsilon^{1.5}})$ overall stochastic first- and second-order oracle queries.


\vspace{-0.2cm}
\subsection{Related Work}
Computing local minimum to a non-convex optimization problem is gaining considerable amount of attentions in recent years.
Both cubic regularization (CR) approaches~\cite{nesterov2006cubic} and trust region (TR) algorithms~\cite{conn2000trust,curtis2017trust} can escape saddle points and find a local minimum by iterating the variable along the direction related to the eigenvector of the Hessian with the most negative eigenvalue.
As the CR heavily depends the regularization parameter for the cubic term, \citet{cartis2011adaptive}~propose an adaptive cubic regularization (ARC) approach to boost the efficiency by adaptively tunes the regularization parameter according to the current objective decrease.
Noting the high cost of full gradient and Hessian computation in ARC, sub-sampled cubic regularization (SCR)~\cite{kohler2017sub} is developed for sampling partial data points to estimate the full gradient and Hessian.
Recently, by exploring the finite-sum structure of the target problem, many works incorporate variance-reduced technique \cite{johnson2013accelerating} into CR and propose stochastic variance-reduced methods. For example, \citet{zhou2018stochastic}~propose stochastic variance-reduced cubic (SVRC) in which they integrate the stochastic variance-reduced gradient estimator~\cite{johnson2013accelerating} and the Hessian tracking technique \cite{pmlr-v84-gower18a} with CR. Such a method is proved to be at least $\OM(n^{1/5})$ faster than CR and TR. Then \citet{zhou2018sample} suggest to use adaptive gradient batch size and constant Hessian batch size, and develop Lite-SVRC to further reduce the stochastic second-order oracle $\OM({n^{4/5}}/{\epsilon^{1.5}})$ of SVRC to $\OM({n^{2/3}}/{\epsilon^{1.5}})$ at the cost of higher gradient computation cost. Similarly, except turning the gradient batch size, \citet{zhang2018adaptive}~further adaptively sample a certain number of data points to estimate the Hessian and prove the proposed method to have the same stochastic second-order oracle complexity as Lite-SVRC.
\begin{table}[t]
	\caption{Stochastic first- and second-order oracle complexities, SFO and SSO for short respectively, of the proposed STR approaches and other state-of-the-arts methods.
		When SSO is prioritized, our STR$_1$ has strictly better complexity than both SCR and Lite-SVRC.
		When SFO an SSO are treated equally, STR$_2$ improves the existing result in SVRC.}\label{tableofcompareresult}
	\centering
	\begin{tabular}{c|c|c}
		\hline
		Algorithm       & SFO & SSO \\ \hline
		TR &   $\OM(\frac{n}{\epsilon^{1.5}})$  &  $\OM(\frac{n}{\epsilon^{1.5}})$   \\
		CR &   $\OM(\frac{n}{\epsilon^{1.5}})$  &  $\OM(\frac{n}{\epsilon^{1.5}})$   \\
		SCR &  $\OM(\frac{1}{\epsilon^{3.5}})$  &  $\OM(\frac{1}{\epsilon^{2.5}})$   \\
		SVRC &   $\OM(\frac{n^{4/5}}{\epsilon^{1.5}})$  &  $\OM(\frac{n^{4/5}}{\epsilon^{1.5}})$   \\
		Lite-SVRC &   $\OM(\frac{n^{2/3}}{\epsilon^{2.5}})$  &  $\OM(\frac{n^{2/3}}{\epsilon^{1.5}})$   \\ \hline
		STR$_1$     &  $\OM(\min\{\frac{n}{\epsilon^{1.5}}, \frac{\sqrt{n}}{\epsilon^2}\})$   &  $\OM(\min\{\frac{1}{\epsilon^2}, \frac{\sqrt{n}}{\epsilon^{1.5}}\})$   \\
		STR$_2$     &  $\OM(\frac{n^{3/4}}{\epsilon^{1.5}})$   &  $\OM(\frac{n^{3/4}}{\epsilon^{1.5}})$   \\ \hline
	\end{tabular}
\end{table}

\section{Preliminary}

\noindent{\bf Notation.}
We use $\|\vB\|$ to denote the Euclidean norm of vector $\vB$ and use $\|\AB\|$ to denote the spectral norm of matrix $\AB$.
Let $\SM$ be the set of component indices.
We define the batch average of component function by
\begin{equation*}
f(\xB; \SM) \defi \frac{1}{|\SM|}\sum_{i\in\SM} f_i(\xB).
\end{equation*}

Then we specify the assumptions that are necessary to the analysis of our methods.
\begin{assumption} \label{ass_boundedness_of_F}
	$F$ is bounded from below and its global optimal is achieved at $\xB^*$.
	We further denote
	\begin{equation*}
	\Delta = F(\xB^0) - F(\xB^*).
	\end{equation*}
\end{assumption}
\begin{assumption} \label{ass_first_order_smooth}
	Each component function $f_i:\RBB^d\rightarrow\RBB$ has $L_1$-Lipschitz continuous gradient: for any $\xB, \yB \in \RBB^d$
	\begin{equation}
		\|\nabla f_i(\xB) - \nabla f_i(\yB)\|\leq L_1\|\xB -\yB\|.
	\end{equation}
\end{assumption}
Clearly, the objective $F$ as the average of $n$ component functions also has $L_1$-Lipschitz continuous gradient.
\begin{assumption} \label{ass_second_order_smooth}
	Each component function $f_i:\RBB^d\rightarrow\RBB$ has $L_2$-Lipschitz continuous Hessian: for any $\xB, \yB \in \RBB^d$
	\begin{equation}
	\|\nabla^2 f_i(\xB) - \nabla^2 f_i(\yB)\|\leq L_2\|\xB -\yB\|.
	\end{equation}
\end{assumption}
Similarly, the objective $F$ has $L_2$-Lipschitz continuous Hessian, which implies the following: for any $\xB, \dB \in \RBB^d$,
$$F(\xB+\dB)\!\leq\! F(\xB)+ \nabla F(\xB)^\top\dB+  \frac{1}{2}\dB^\top\nabla^2F(\xB)\dB + \frac{L_2}{6}\|\dB\|^3.$$
\subsection{Trust Region Method}
The trust region method has a long history \cite{conn2000trust}.
In each step, it solves the Quadratic Constraint Quadratic Program (QCQP)
\begin{equation}
\hB^{k} := \argmin_{\hB\in\RBB^d, \|\hB\|\leq r} \langle \nabla F(\xB^k), \hB\rangle + \frac{1}{2}\langle\nabla^2 F(\xB^k)\hB, \hB\rangle,
\label{eqn_qcqp}
\end{equation}
where $r$ is the trust-region radius, and updates
\begin{equation}
\xB^{k+1} := \xB^{k} + \hB^{k}.
\label{eqn_trust_region_update}
\end{equation}
Since $\nabla^2 F(\xB^k)$ is indefinite, the trust-region subproblem (\ref{eqn_qcqp}) is non-convex, but its global optimizer can be characterized by the following lemma.
\begin{lemma}[Corollary 7.2.2 in \cite{conn2000trust}]
	Any global minimizer of problem (\ref{eqn_qcqp}) satisfies the equation
	\begin{equation}
		\left(\nabla^2 F(\xB^k) + \lambda \IB\right)\hB^{k} = -\nabla F(\xB^{k}),
	\end{equation}
	where the dual variable $\lambda \geq 0$ should satisfy $\nabla^2 F(\xB^k) + \lambda \IB \succcurlyeq 0$ and $\lambda(\|\hB^{k}\| - r)=0$.
	\label{lemma_qcqp_global_solution}
\end{lemma}
In particular, the standard QCQP solver returns both the minimizer $\hB^{k}$ as well as the corresponding dual variable $\lambda$ of subproblem \eqref{eqn_qcqp}.
While it is known that the trust-region update (\ref{eqn_qcqp}) and (\ref{eqn_trust_region_update}) converges at the rate $\OM({1}/{\sqrt{k}})$, recently \citep{curtis2017trust} proposes a trust-region variant which converges at the optimal rate $\OM({1}/{k^{2/3}})$ \cite{carmon2017lower}.
In this paper, we show that the vanilla trust-region update (\ref{eqn_qcqp}) and (\ref{eqn_trust_region_update}) already achieves the optimal convergence rate as the byproduct of our novel argument.

\begin{algorithm}[t]
	\floatname{algorithm}{MetaAlgorithm}
	\caption{Inexact Trust Region Method}
	\label{alg_stochastic_trust_region}
	\begin{algorithmic}[1]
		\REQUIRE Initialization $\xB^{0}$, step size $r$, number of iterations $K$, construction of differential estimators $\gB^{k}$ and $\HB^{k}$
		\FOR{$k = 1$ {\bfseries to} $K$}
		\STATE Compute $\hB^{k}$ and $\lambda^{k}$ by solving (\ref{eqn_qcqp_inexact});
		\STATE $\xB^{k+1} := \xB^{k} + \hB^{k}$;
		\IF{$\lambda^{k} \leq {2\sqrt{\epsilon/L_2}}$}
		\STATE Output $\xB_\epsilon = \xB^{k+1}$; \label{eqn_termination}
		\ENDIF
		\ENDFOR
	\end{algorithmic}
\end{algorithm}
\section{Methodology}
In this section, we first introduce a general inexact trust region method which is summarized in MetaAlgorithm~\ref{alg_stochastic_trust_region}.
It accepts inexact gradient estimation $\gB^{k}$ and Hessian estimation $\HB^{k}$ as input to the QCQP subproblem
\begin{equation}
	\hB^{k} := \argmin_{\hB\in\RBB^d, \|\hB\|\leq r} \langle \gB^{k}, \hB\rangle + \frac{1}{2}\langle\HB^{k}\hB, \hB\rangle.
	\label{eqn_qcqp_inexact}
\end{equation}
Similar to (\ref{eqn_qcqp}), the global solutions to (\ref{eqn_qcqp_inexact}) are characterized by Lemma \ref{lemma_qcqp_global_solution} and we further denote the dual variable corresponding to the minimizer $\hB^{k}$ by $\lambda^{k}$.
In practice, (\ref{eqn_qcqp_inexact}) can be efficiently solved by Lanczos method~\cite{gould1999solving}.

We prove that such inexact trust-region method achieves the optimal $\OM({1}/{k^{2/3}})$ convergence rate when the estimation $\gB^{k}$ and $\HB^{k}$ at each iteration are sufficient close to their full (exact) counterparts $\nabla F(\xB^{k})$ and $\nabla^2 F(\xB^{k})$ respectively:
\begin{equation}\label{conditiongradient}
\|\gB^{k} - \nabla F(\xB^{k})\| \leq \frac{\epsilon}{6}, \
\|\HB^{k} - \nabla^2 F(\xB^{k})\|\leq \frac{\sqrt{\epsilon L_2}}{3}.
\end{equation}

Such result allows us to derive stochastic trust-region variants with novel differential estimators that are tailored to ensure the optimal convergence rate.
We state our formal results in Theorem~\ref{thm_TR_meta}.


\begin{theorem}
	Consider problem (\ref{eqn_stochastic_problem}) under Assumption \ref{ass_boundedness_of_F}-\ref{ass_second_order_smooth}.
	If the differential estimators $\gB^{k}$ and $\HB^{k}$ satisfy Eqn.~\eqref{conditiongradient} for all $k$, MetaAlgorithm \ref{alg_stochastic_trust_region} finds an $\OM(\epsilon, \sqrt{\epsilon})$-SOSP in less than $K = \OM(\sqrt{L_2}\Delta/\epsilon^{1.5})$ iterations by setting the trust-region radius as $r = \sqrt{\epsilon/L_2}$.
	\label{thm_TR_meta}
\end{theorem}
\begin{proof}
	For simplicity of notation, we denote $$\nabla_k \defi \gB^{k} - \nabla F(\xB^{k}) \text{ and } \nabla_k^2 \defi \HB^{k} - \nabla^2 F(\xB^{k}).$$
	From Assumption \ref{ass_second_order_smooth} we have
	\begin{align*}
	&F(\xB^{k+1}) \\
	\leq& F(\xB^{k}) \!+\! \langle\nabla F(\xB^{k}), \hB^{k}\rangle  + \frac{1}{2}\langle\nabla^2 F(\xB^{k})\hB^{k}, \hB^{k}\rangle + \frac{L_2}{6}\|\hB^{k}\|^3\\
	=& F(\xB^{k})  \!+\! \langle \nabla_k \!+\! \gB^{k}, \hB^{k}\rangle + \frac{1}{2}\langle[\nabla_k^2+\HB^{k}]\hB^{k}, \hB^{k}\rangle \!+\! \frac{L_2}{6}\|\hB^{k}\|^3.
	\end{align*}
	Use the Cauchy–Schwarz inequality to obtain
	\begin{align}
	F(\xB^{k+1}) \leq& F(\xB^{k}) \!+\! \langle\gB^{k}, \hB^{k}\rangle \!+\! \frac{1}{2}\langle\HB^{k}\hB^{k}, \hB^{k}\rangle\!+\! \frac{L_2}{6}\|\hB^{k}\|^3 \notag \\
	& \!+\! \|\nabla_k\|\|\hB^{k}\| \!+\! \frac{1}{2}\|\nabla_k^2\|\|\hB^{k}\|^2.
	\label{eqn_proof_4_stochastic}
	\end{align}
	The requirement (\ref{conditiongradient}) together with the trust region $\|\hB\|\leq\sqrt{\epsilon/L_2}$ allow us to bound
	\begin{align}
	\|\nabla_k\|\|\hB^{k}\| \!+\! \frac{1}{2}\|\nabla_k^2\|\|\hB^{k}\|^2 \leq \frac{1}{3}\cdot\frac{\epsilon^{1.5}}{\sqrt{L_2}}.
	\label{eqn_proof_0_stochastic}
	\end{align}	
	The optimality of (\ref{eqn_qcqp}) indicates that there exists dual variable $\lambda^{k} \geq 0$ so that (Corollary 7.2.2 in \cite{conn2000trust})
	\begin{align}
	\text{First Order}&:\gB^{k} + \HB^{k}\hB^{k} + \frac{\lambda^{k} L_2}{2} \hB^{k} = 0, \label{eqn_optimality_first_order_stochastic} \\
	\text{Second Order}&: \HB^{k} + \frac{\lambda^{k} L_2}{2} \cdot\IB \succcurlyeq 0, \label{eqn_optimality_second_order_stochastic} \\
	\text{Complementary}&: \lambda^{k} \cdot (\|\hB^{k}\| - r) = 0.
	\label{eqn_optimality_complementary_stochastic}
	\end{align}
	Multiplying (\ref{eqn_optimality_first_order_stochastic}) by $\hB^{k}$, we have
	\begin{equation}
	\langle\gB^{k} + \HB^{k}\hB^{k} + \frac{\lambda^{k} L_2}{2} \hB^{k}, \hB^{k}\rangle = 0.
	\label{eqn_proof_1_stochastic}
	\end{equation}
	Additionally, using (\ref{eqn_optimality_second_order_stochastic}) we have $$\langle (\HB^{k} + \frac{\lambda^{k} L_2}{2})\hB^{k}, \hB^{k}\rangle\geq 0,$$ which together with (\ref{eqn_proof_1_stochastic}) gives
	\begin{equation}
	\langle\gB^{k}, \hB^{k}\rangle \leq 0.
	\label{eqn_proof_2_stochastic}
	\end{equation}
Moreover, the complementary property \eqref{eqn_optimality_complementary_stochastic} indicates $\|\hB^{k}\| $ $=\sqrt{{\epsilon}/{L_2}}$ as we have $\lambda^{k}\geq 2\sqrt{\epsilon/L_2} >0$ in MetaAlgorithm.
	Plug (\ref{eqn_proof_0_stochastic}), (\ref{eqn_proof_1_stochastic}), and (\ref{eqn_proof_2_stochastic}) into (\ref{eqn_proof_4_stochastic}) and use $\|\hB\|=\sqrt{{\epsilon}/{L_2}}$:
	\begin{equation}
	F(\xB^{k+1}) \leq F(\xB^{k}) - \frac{L_2\lambda^{k}}{4}\cdot\frac{\epsilon}{L_2} + \frac{1}{3}\cdot\frac{\epsilon^{1.5}}{\sqrt{L_2}}.
	\label{eqn_proof_3_stochastic}
	\end{equation}
	Therefore, if we have $\lambda^{k}>{2\epsilon^{0.5}}/{\sqrt{L_2}}$, then
	\begin{equation}
	F(\xB^{k+1})\leq F(\xB^{k}) - {\frac{1}{6\sqrt{L_2}}}\cdot\epsilon^{1.5}.
		\label{eqn_proof_4}
	\end{equation}
	Using Assumption \ref{ass_boundedness_of_F}, we find $\lambda^{k} \leq {2\epsilon^{0.5}}/{\sqrt{L_2}}$ in no more than $6{\sqrt{L_2}\cdot(F(\xB^{0}) - F(\xB^*))}/{\epsilon^{1.5}}$ iterations.
	
	We now show that once $\lambda^{k} \leq {2\epsilon^{0.5}}/{\sqrt{L_2}}$, then $\xB^{k+1}$ is already an $\OM(\epsilon)$-SOSP:
	From (\ref{eqn_optimality_first_order_stochastic}), we have
	\begin{equation}
	\|\gB^{k} + \HB^{k}\hB^{k}\| = \frac{L_2 \lambda^{k}}{2}\cdot\|\hB^{k}\| \leq 2 \epsilon.
	\end{equation}
	The assumptions $\|\nabla_k\|\leq \epsilon/6$ and $\|\nabla_k^2\|\leq \sqrt{\epsilon L_2}/3$ together with the trust region $\|\hB\|\leq\sqrt{\epsilon/L_2}$ imply
	\begin{equation}
	\begin{aligned}
	&\|\nabla F(\xB^{k}) + \nabla^2 F(\xB^{k})\hB^{k}\| \\
	\leq &\|\gB^{k} + \HB^{k}\hB^{k}\| +\|\nabla_k\| + \|\nabla_k^2\cdot\hB^{k}\| \leq 2.5\epsilon.
	\end{aligned}
	\end{equation}
	On the other hand use Assumption \ref{ass_second_order_smooth} to bound
	\begin{equation*}
	\|\nabla F(\xB^{k+1}) - \nabla F(\xB^{k}) - \nabla^2 F(\xB^{k})\hB^{k}\|\leq \frac{L_2}{2}\|\hB^{k}\|^2 \leq \frac{\epsilon}{2}.
	\end{equation*}
	Combining these two results gives $\|\nabla F(\xB^{k+1})\| \leq 3\epsilon$.\\
	Besides use Assumption \ref{ass_second_order_smooth}, $\|\nabla^2_k\|\leq \sqrt{\epsilon L_2}/3$, and (\ref{eqn_optimality_second_order_stochastic}) we derive the Hessian lower bound
	\begin{equation*}
	\begin{aligned}
	\nabla^2 F(\xB^{k+1})& \succcurlyeq \nabla^2 F(\xB^{k}) - L_2\cdot\|\hB^{k}\|\IB \\
	&\succcurlyeq \HB^{k} \!-\! \sqrt{\epsilon L_2}/3\IB \!-\! L_2\|\hB^{k}\|\IB \succcurlyeq \!-\! \frac{10}{3}\sqrt{L_2\epsilon}\IB.
	\end{aligned}
	\end{equation*}
	Hence $\xB^{k+1}$ is a $12\epsilon$-stationary point.
	
	Therefore, we have $\|\hB^k\| = r$ according to the complementary condition (\ref{eqn_optimality_complementary_stochastic}) for all but the last iteration.
\end{proof}
\begin{remark}
	We emphasize that MetaAlgorithm \ref{alg_stochastic_trust_region} degenerates to the exact trust region method by taking $\gB^{k} = \nabla F(\xB^k)$ and $\HB^{k} = \nabla^2 F(\xB^k)$.
	Such result is of its own interest because this is the first proof to show that the vanilla trust region method has the optimal $\OM({1}/{k^{2/3}})$ convergence rate.
	Similar rate is achieved by \cite{curtis2017trust} but with a complicated trust region variant.
\end{remark}
\begin{remark}
	We note that MetaAlgorithm \ref{alg_stochastic_trust_region} uses the dual variable $\lambda^{k}$ as stopping criterion, which enables the last-term convergence analysis in Theorem \ref{thm_TR_meta}.
	In our appendix, we present a variant of MetaAlgorithm \ref{alg_stochastic_trust_region} without accessing to the exact dual variable.
	Further, we show such variant enjoys a similar convergence guarantee in expectation.
\end{remark}

Theorem~\ref{thm_TR_meta} shows the explicit step size control of the trust-region method:
Since the dual variable satisfies $\lambda^k > {2\epsilon^{0.5}}/{\sqrt{L_2}} >0$ for all but the last iteration, we always find the solution to the trust-region subproblem (\ref{eqn_qcqp_inexact}) in the boundary, i.e. $\|\hB^k\| = r$, according to the complementary condition (\ref{eqn_optimality_complementary_stochastic}).
Such exact step-size control property is missing in the cubic-regularzation method where the step-size is implicitly decided by the cubic regularization parameter.

More importantly, we emphasize that such explicit step size control is crucial to the sample efficiency of our variance reduced differential estimators.
The essence of variance reduction is to exploit the correlations between the differentials in consecutive iterations.
Intuitively, when two neighboring iterates are close, so are their differentials due to the Lipschitz continuity, and hence a smaller number of samples suffice to maintain the accuracy of the estimators.
On the other hand, smaller step size reduces the per-iteration objective decrease which harms the convergence rate of the algorithm (see proof of Theorem~\ref{thm_TR_meta}).
Therefore, the explicit step-size control in trust-region method allows us to well trade-off the per-iteration sample complexity and convergence rate, from which we can derive stochastic trust region approaches with state-of-the-art sample efficiency.

\section{Stochastic Trust Region Method: Type I}
\begin{algorithm}[t]
	\caption{STR$_1$}
	\label{alg_str_1}
	\begin{algorithmic}[1]
		\REQUIRE Initializer $\xB^{0}$, step size $r$, number of iterations $K$
		\FOR{$k = 1$ {\bfseries to} $K$}
		\STATE Construct gradient estimator $\gB^{k}$ by Estimator \ref{alg_SPIDER_estimator};
		\STATE Construct Hessian estimator $\HB^{k}$ by Estimator \ref{alg_HB_estimator};
		\STATE Compute $\hB^{k}$ by solving (\ref{eqn_qcqp_inexact});
		\STATE $\xB^{k+1} := \xB^{k} + \hB^{k}$;
		\IF{$\lambda^{k} \leq {2\sqrt{\epsilon/L_2}}$}
		\STATE Output $\xB_\epsilon = \xB^{k+1}$;
		\ENDIF
		\ENDFOR
	\end{algorithmic}
\end{algorithm}
Having the inexact trust-region method as prototype, we now derive our first sample-efficient stochastic trust region methods, namely STR$_1$, which emphasises more cheaper stochastic second-order oracle complexity.
As Theorem \ref{thm_TR_meta} already guarantees the optimal convergence rate of MetaAlgorithm \ref{alg_stochastic_trust_region} when   the gradient estimator $\gB^{k}$ and the Hessian estimator $\HB^{k}$ meet requirements (\ref{conditiongradient}), here we focus on constructing such novel differential estimators. Specifically,
we first present our Hessian estimator in Estimator \ref{alg_HB_estimator} and our first gradient estimator in Estimator \ref{alg_SPIDER_estimator}, both of which exploit the trust region radius $r=\sqrt{\epsilon L_2}$ to reduce their variances.
Further, by plugging Estimator \ref{alg_HB_estimator} and Estimator \ref{alg_SPIDER_estimator} in MetaAlgorithm \ref{alg_stochastic_trust_region}, we present STR$_1$ in Algorithm \ref{alg_str_1} with state-of-the-art stochastic Hessian complexity.

\subsection{Hessian Estimator}
\begin{algorithm}[t]
	\floatname{algorithm}{Estimator}
	\caption{Hessian Estimator}
	\label{alg_HB_estimator}
	\begin{algorithmic}[1]
		\REQUIRE Epoch length $p_2$, sample size $s_2$, $s_2'$ (optional)
		\IF{mod($k, p_2$)$=0$}
		\STATE Option I: \ \qquad $\diamond$ high accuracy case (small $\epsilon$)  \\
		\quad $\HB^{k} := \nabla^2 F(\xB^{k})$;\label{step2}
		\STATE Option II: \qquad $\diamond$ low accuracy case (moderate $\epsilon$)\\
		\quad Draw $s_2'$ samples indexed by $\HM'$;\\
		\quad $\HB^{k} := \nabla^2 f(\xB^{k}; \HM')$;
		\ELSE
		\STATE Draw $s_2$ samples indexed by $\HM$;
		\STATE $\HB^{k} := \nabla^2 f(\xB^{k}; \HM)-\nabla^2 f(\xB^{k-1}; \HM)+\HB^{k-1}$;
		\ENDIF	
	\end{algorithmic}
\end{algorithm}

Our epoch-wise Hessian estimator $\HB^{k}$ is given in Estimator \ref{alg_HB_estimator}, where $p_2$ controls the epoch length and $s_2$  (and optionally $s_2'$) controls the minibatch size.
At the beginning of each epoch Estimator \ref{alg_HB_estimator} has two options, designed for different target accuracy:
Option I is preferable for the high accuracy case ($\epsilon < \OM(1/n)$) where we compute the full Hessian to avoid approximation error, and Option II is designed for the moderate accuracy case ($\epsilon > \OM(1/n)$) where we only need to an approximate Hessian  estimator.
Then, $p_2$ iterations follow with $\HB^{k}$ defined in a recurrent manner.
 These recurrent estimators exist for the first-order case \cite{nguyen2017sarah,fang2018spider}, but their bound only holds under the vector $\ell_2$ norm. Here we generalize them into Hessian estimation with matrix spectrum norm bound.

The following lemma analyzes the amortized stochastic second-order oracle (Hessian) complexity for Algorithm \ref{alg_HB_estimator} to meet the requirement in Theorem \ref{thm_TR_meta}.
As we need an approximation error bound under the spectrum norm, we will appeal to the matrix Azuma's inequality \cite{tropp2012user}.
\begin{lemma}\label{lemma_Hessian_estimator}
	Assume Algorithm \ref{alg_str_1} takes the trust region radius $r = \sqrt{\epsilon/L_2}$ as in Theorem \ref{thm_TR_meta}.
	For any $k\geq 0$, Estimator \ref{alg_HB_estimator} produces estimators $\HB^{k}$ for the second order differentials $\nabla^2 F(\xB^{k})$ such that $\|\HB^{k} - \nabla^2 F(\xB^{k})\|\leq \sqrt{\epsilon L_2}/3$ with probability at least $1-\delta/K_0$ if we set\\
	1. $p_2 = \sqrt{n}$ and $s_2 = 32\sqrt{n}\log({dK_0}/{\delta})$ in option I, or\\
	2. $p_2 = {L_1}/({2\sqrt{\epsilon L_2}})$, $s_2'={16L_1^2}/{(\epsilon L_2)}\log({dK_0}/{\delta})$, and $s_2 = {32L_1}/{(\sqrt{\epsilon L_2})}\log({dK_0}/{\delta})$ in option II .\\
	Consequently the amortized per-iteration stochastic second-order oracle complexity to construct $\HB^{k}$ is no more than $$2s_2 =\min\{ 64\sqrt{n}\log\frac{d}{\delta K_0}, \frac{64L_1}{\sqrt{\epsilon L_2}}\log\frac{dK_0}{\delta}\}.$$
\end{lemma}

\begin{proof}
	Without loss of generality, we analyze the case $0\leq k< q_2$ for ease of notation.
	We first focus on Option II.
	The proof for Option I follows the similar argument.\\
	\noindent{\bf Option II:}
	Define for $k= 0$ and $i\in[s_2']$ $$\BB_i^0 \defi \nabla^2 f_i(\xB^{0}) - \nabla^2 F(\xB^{0}),$$
	and define for $k\geq 1$ and $i\in[s_2]$ $$\BB_i^k \defi \nabla^2 f_i(\xB^{k}) - \nabla^2 f_i(\xB^{k-1}) - (\nabla^2 F(\xB^{k}) - \nabla^2 F(\xB^{k-1})).$$
	$\{\BB_i^k\}$ is a martingale difference.
	We have for all $k$ and $i$,
	\begin{equation*}
	\EBB[\BB_i^k|\xB^k] = 0.
	\end{equation*}
	Besides, using Assumption \ref{ass_first_order_smooth} for $k = 0$ to bound
	\begin{equation}
	\|\BB_i^0\| \leq \|\nabla^2 f_i(\xB^{0})\| +\|\nabla^2 F(\xB^{0})\| = 2L_1,
	\end{equation}
	and using Assumption \ref{ass_second_order_smooth} for $k\geq 1$ to bound
	\begin{equation*}
	\begin{aligned}
	\|\BB_i^k\| \leq& \|\nabla^2 f_i(\xB^{k}) - \nabla^2 f_i(\xB^{k-1})\| \\
	&+ \|\nabla^2 F(\xB^{k}) - \nabla^2 F(\xB^{k-1})\| \leq 2\sqrt{\epsilon L_2}.
	\end{aligned}
	\end{equation*}
	From the construction of $\HB^{k}$, we have $$\HB^k -\nabla^2 F(\xB^{k})= \sum_{i=1}^{s_2'}\frac{\BB_i^{0}}{s_2'} + \sum_{j=1}^{k}\sum_{i=1}^{s_2} \frac{\BB_i^{j}}{s_2}.$$
	Thus using the matrix Azuma's Inequality in Theorem 7.1 of \cite{tropp2012user} and $k\leq p_2$, we have
	\begin{equation*}
	\begin{aligned}
	&Pr\{\|\HB^k -\nabla^2 F(\xB^{k})\| \geq t\} \\
	\leq &d\cdot\exp\{ -\frac{t^2/8}{\sum_{i=1}^{s_2'}{4L_1^2}/{s_2'^2} + \sum_{j=1}^{k}\sum_{i=1}^{s_2}{4\epsilon L_2}/{s_2^2}}\}\\
	\leq &d\cdot\exp\{ -\frac{t^2/8}{{4L_1^2}/{s_2'} + 4p_2{\epsilon L_2}/{s_2}}\}.
	\end{aligned}
	\end{equation*}
	Consequently, we have
	\begin{equation*}
	\begin{aligned}
	&Pr\{\|\HB^k -\nabla^2 F(\xB^{k})\| \leq \sqrt{\epsilon L_2}\} \geq 1 - \delta/K_0.
	\end{aligned}
	\end{equation*}
	by taking $t=\sqrt{\epsilon L_2}$, $s_2'={16L_1^2}/{(\epsilon L_2)}\log({dK_0}/{\delta})$, $s_2 = {32L_1}/({\sqrt{\epsilon L_2}})\log({dK_0}/{\delta})$, and $p_2 = {L_1}/({2\sqrt{\epsilon L_2}})$.
	
	\noindent{\bf Option I:} The proof is similar to the one of Option II except that we replace $\BB_i^0$ with zero matrix.
	In such case, the matrix Azuma's Inequality implies
	\begin{equation*}
	\begin{aligned}
	&Pr\{\|\HB^k -\nabla^2 F(\xB^{k})\| \geq t\} \\
	\leq &d\cdot\exp\{ -\frac{t^2/8}{\sum_{j=1}^{k}\sum_{i=1}^{s_2}{4\epsilon L_2}/{s_2^2}}\}\leq d\cdot\exp\{ -\frac{t^2/8}{4p_2\epsilon L_2/s_2}\}.
	\end{aligned}
	\end{equation*}		
	Thus by taking $t=\sqrt{\epsilon L_2}$, $s_2 = 32\sqrt{n}\log({d}/{\delta})$, and $p_2 = \sqrt{n}$, we have the result.
	
	\noindent{\bf Amortized Complexity:}
	In option I, the choice of parameter ensures that: $s_2' \leq p_2\times s_2$ and in option II: $n \leq p_2\times s_2$.
	Consequently the amortized stochastic second-order oracle is bounded by $2s_2$.
\end{proof}

\subsection{Gradient Estimator: Case (1)}\label{resultsofcomplexity}
\begin{algorithm}[t]
	\floatname{algorithm}{Estimator}
	\caption{Gradient Estimator: Case (1)}
	\label{alg_SPIDER_estimator}
	\begin{algorithmic}[1]
		\IF{mod($k,p_1$)$=0$}
		\STATE $\gB^{k} := \nabla F(\xB^k)$
		\ELSE
		\STATE Draw $s_1$ samples indexed by $\GM$;
		\STATE $\gB^{k}= \nabla f(\xB^{k}; \GM)-\nabla f(\xB^{k-1}; \GM) + \gB^{k-1}$;
		\ENDIF
	\end{algorithmic}
\end{algorithm}
When stochastic second-order oracle complexity is prioritized, we directly employ the SPIDER gradient estimator to construct $\gB^k$ \cite{fang2018spider}.
Similar to the construction for $\HB^k$, the estimator $\gB^k$ is also construct in an epoch-wise manner as presented in Estimator \ref{alg_SPIDER_estimator}, where $p_1$ controls the epoch length and $s_1$ controls the minibatch size.

We now analyze the necessary stochastic first-order oracle complexity to meet the requirement in Theorem \ref{thm_TR_meta}.
\begin{lemma} \label{lemma_SPIDER}
	Assume Algorithm \ref{alg_str_1} takes the trust region radius $r = \sqrt{\epsilon/L_2}$.
	Estimator \ref{alg_SPIDER_estimator} produces estimator $\gB^{k}$ of the first order differential $\nabla F(\xB^{k})$ such that $\|\gB^{k} - \nabla F(\xB^{k})\|\leq \epsilon/6$ with probability at least $1-\delta/K_0$ for any $k\geq 0$, if we set $p_1 = \max\{1, \sqrt{{n\epsilon L_2}/({cL_1^2 \log\frac{K_0}{\delta}})}\}$ and $s_1 = \min\{n, \sqrt{{cnL_1^2 \log({K_0}/{\delta})}/({\epsilon L_2})}\}$, where the constant $c = 1152$.
	Consequently, the amortized per-iteration stochastic first-order oracle complexity to construct $\gB^{k}$ is $\min\{n, \sqrt{{ 4cnL_1^2 \log{(K_0}/{\delta})}/({\epsilon L_2})}\}$.
\end{lemma}
The proof for Lemma \ref{lemma_SPIDER} is similar to the one of Lemma \ref{lemma_Hessian_estimator} and is deferred to Appendix \ref{section_proof_lemma_spider}.\\
Lemma \ref{lemma_SPIDER} and Lemma \ref{lemma_Hessian_estimator} only guarantee the differential estimators satisfy the requirement \eqref{conditiongradient} in a single iteration and can be extended to hold for all $k$ by using the union bound with $K_0 = 2K$.
Combining such lifted result with Theorem \ref{thm_TR_meta}, we can establish the bound of computational complexity as follows.
\begin{corollary}
	Assume Algorithm \ref{alg_str_1} will use Estimator \ref{alg_SPIDER_estimator} to construct the first-order differential estimator $\gB^k$ and use Estimator \ref{alg_HB_estimator} to construct the second-order differential estimator $\HB^k$.
	To find an $12\epsilon$-SOSP with probability at least $1-\delta$, the overall stochastic first-order oracle complexity is $\min\{{6n\sqrt{L_2}\Delta}/{\epsilon^{1.5}}, {7000\sqrt{n}L_1}/{\epsilon^2}\log({L_2}/{\delta \epsilon})\}$ and the overall stochastic second-order oracle complexity is $\OM(\min\{{\sqrt{nL_2}\Delta}/{\epsilon^{1.5}}\log({dL_2}/{\delta\epsilon}), {L_1}/{\epsilon^2}\log({L_2d}/{\delta \epsilon}) \})$.
	\label{corollary_str_1}
\end{corollary}
From Corollary \ref{corollary_str_1} we see that $\tilde{\OM}(\min\{\sqrt{n}/{\epsilon^{1.5}}, 1/{\epsilon^2}\})$ stochastic second-order oracle queries are sufficient for STR$_1$ to find an $\epsilon$-SOSP which is significantly better than both the subsampled cubic regularization method $\OM(1/\epsilon^{2.5})$ \cite{kohler2017sub} and the variance reduction based ones $\OM(n^{2/3}/\epsilon^{1.5})$ \cite{zhou2018sample,zhang2018adaptive}.
\section{Stochastic Trust Region Method: Type II}
In the previous section, we focus on the setting where the stochastic second-order oracle complexity is prioritized over the stochastic first-order oracle complexity and STR$_1$ achieves the state-of-the-art efficiency.
In this section, we consider a different complexity measure where the first-order and second-order oracle complexities are treated equally and our goal is to minimize the maximum of them.
We note that, currently the best result is $\OM(n^{4/5}/\epsilon^{1.5})$ of the SVRC method \cite{zhou2018stochastic}.

Since the Hessian estimator $\HB^k$ of STR$_1$ already delivers the superior $\OM(\sqrt{n}/\epsilon^{1.5})$ stochastic Hessian complexity, in STR$_2$ (see Algorithm \ref{alg_str_2}), we retain Estimator \ref{alg_HB_estimator} for second-order differential estimation and use Estimator \ref{alg_CASPIDER_estimator} to further reduce the stochastic gradient complexity.
\begin{algorithm}[t]
	\caption{STR$_2$}
	\label{alg_str_2}
	\begin{algorithmic}[1]
		\REQUIRE Initializer $\xB^{0}$, step size $r$, number of iterations $K$
		\FOR{$k = 1$ {\bfseries to} $K$}
		\STATE Construct gradient estimator $\gB^{k}$ by Estimator \ref{alg_CASPIDER_estimator};
		\STATE Construct Hessian estimator $\HB^{k}$ by Estimator \ref{alg_HB_estimator};
		\STATE Compute $\hB^{k}$ by solving (\ref{eqn_qcqp_inexact});
		\STATE $\xB^{k+1} := \xB^{k} + \hB^{k}$;
		\IF{$\lambda^{k} \leq {2\sqrt{\epsilon/L_2}}$}
		\STATE Output $\xB_\epsilon = \xB^{k+1}$;
		\ENDIF
		\ENDFOR
	\end{algorithmic}
\end{algorithm}
\begin{algorithm}[t]
	\floatname{algorithm}{Estimator}
	\caption{Gradient Estimator: Case (2)}
	\label{alg_CASPIDER_estimator}
	\begin{algorithmic}[1]
		\IF{mod($k,p_1$)$=0$}
		\STATE Let $\tilde{\xB} := \xB^{k}$
		\STATE $\gB^{k} := \nabla F(\tilde{\xB})$
		\ELSE
		\STATE Draw $s_1$ samples indexed by $\GM$;
		\STATE $\cB^k = [\nabla^2 F(\tilde{\xB}) -\nabla^2 f_i(\tilde{\xB}; \GM)](\xB^k - \xB^{k-1})$;
		\STATE $\gB^{k}= \nabla f(\xB^{k}; \GM)-\nabla f(\xB^{k-1}; \GM) + \gB^{k-1} + \cB^k$;
		\ENDIF
	\end{algorithmic}
\end{algorithm}

\begin{figure*}[t]
	\begin{center}
		\setlength{\tabcolsep}{0.8pt} 
		\begin{tabular}{cccc}
			\includegraphics[width=0.245\linewidth]{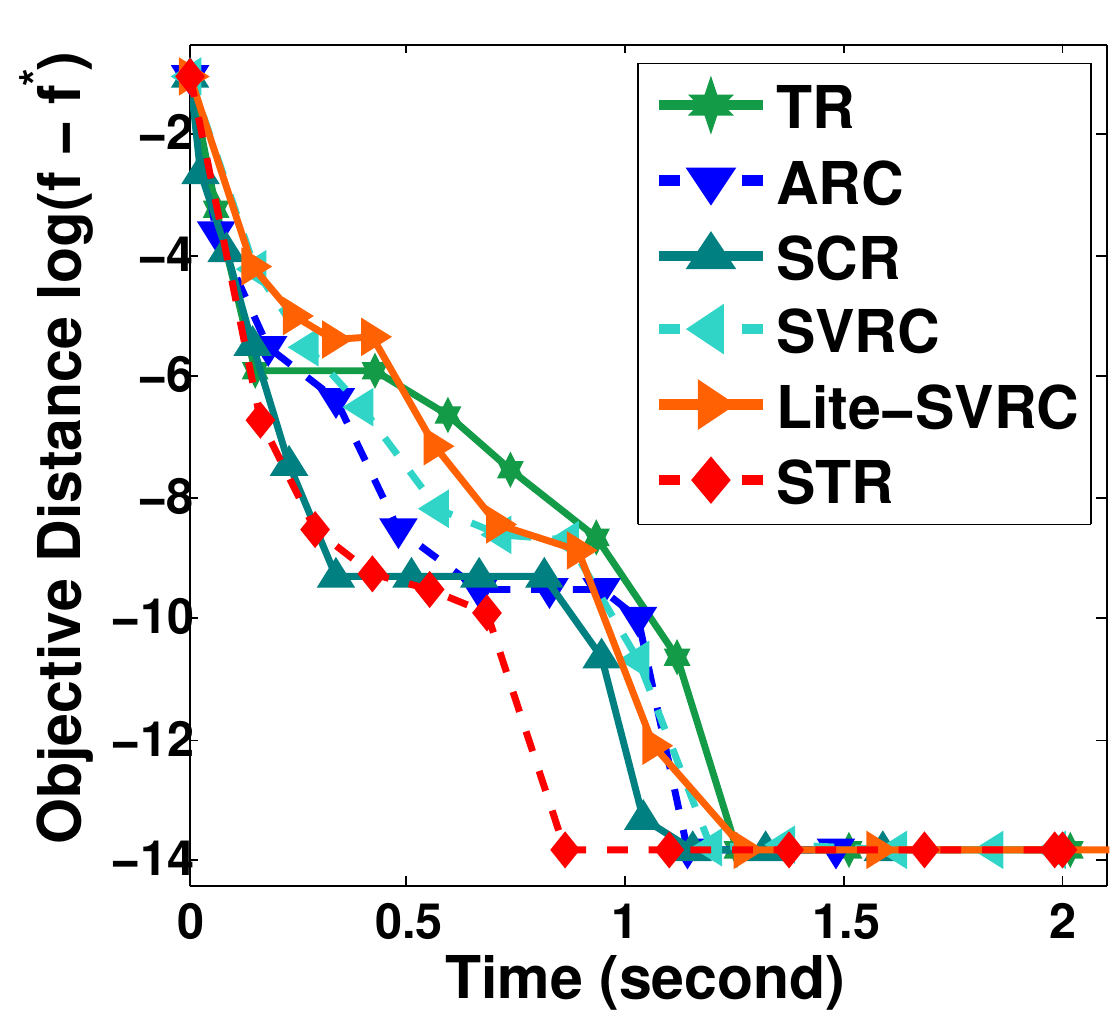}&
			\includegraphics[width=0.245\linewidth]{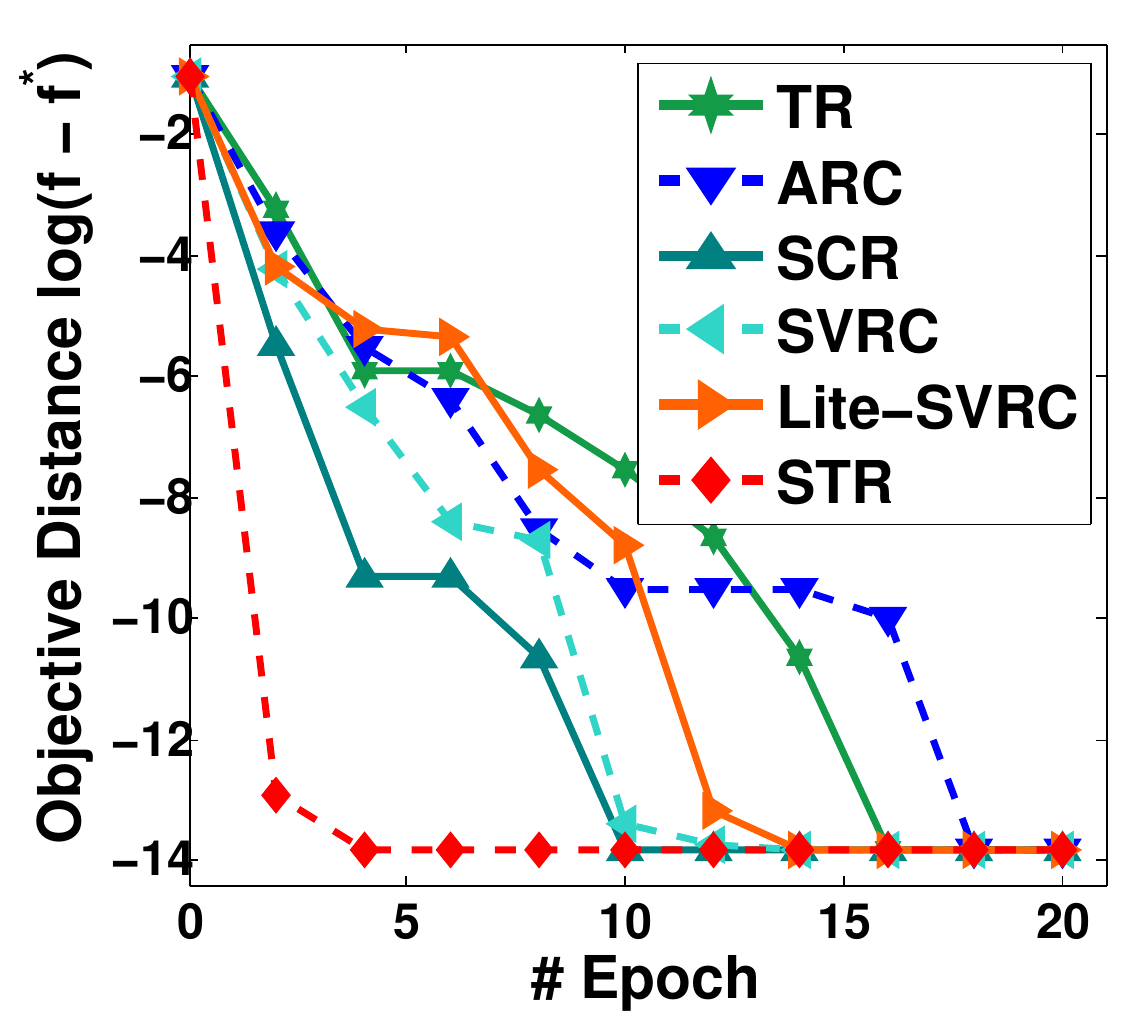}&
			\includegraphics[width=0.245\linewidth]{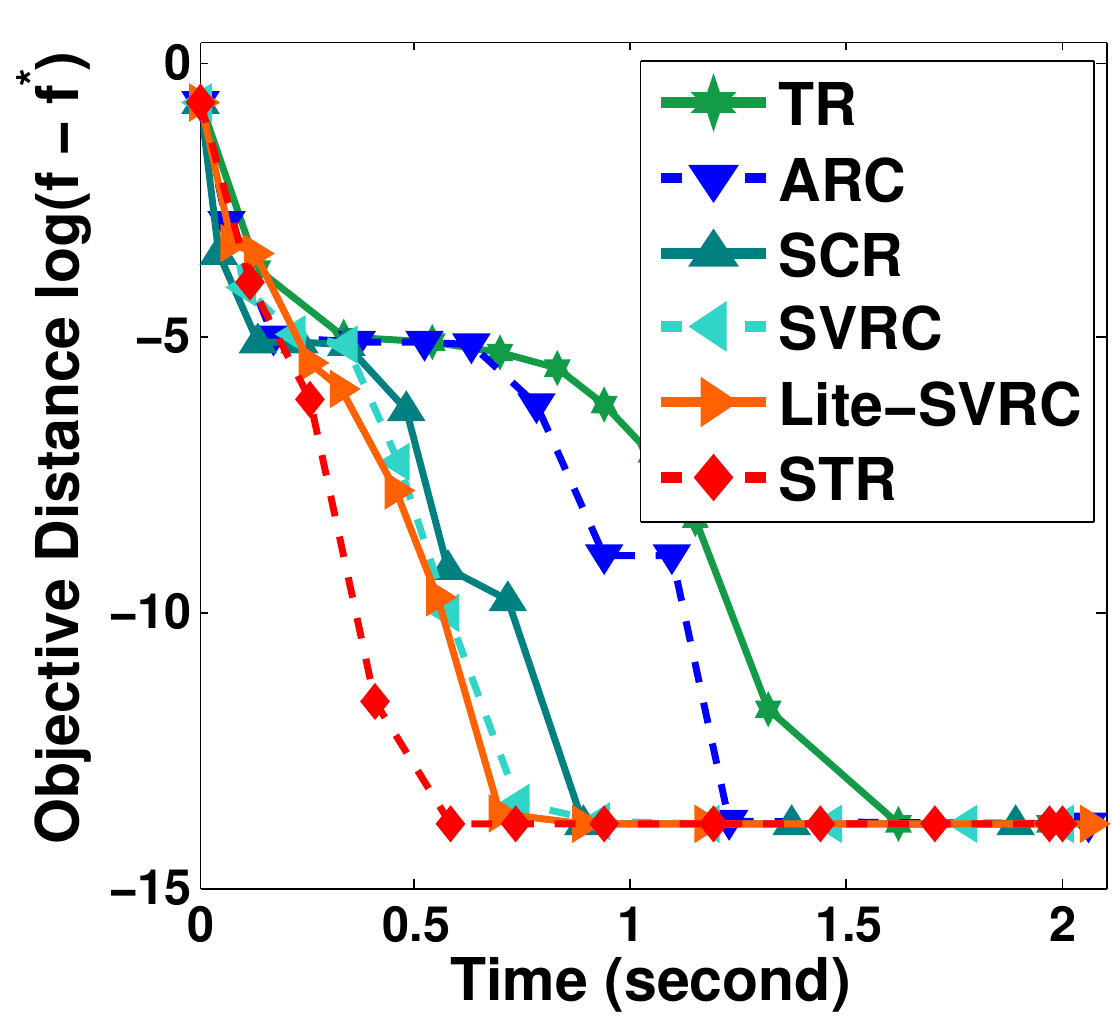}&
			\includegraphics[width=0.245\linewidth]{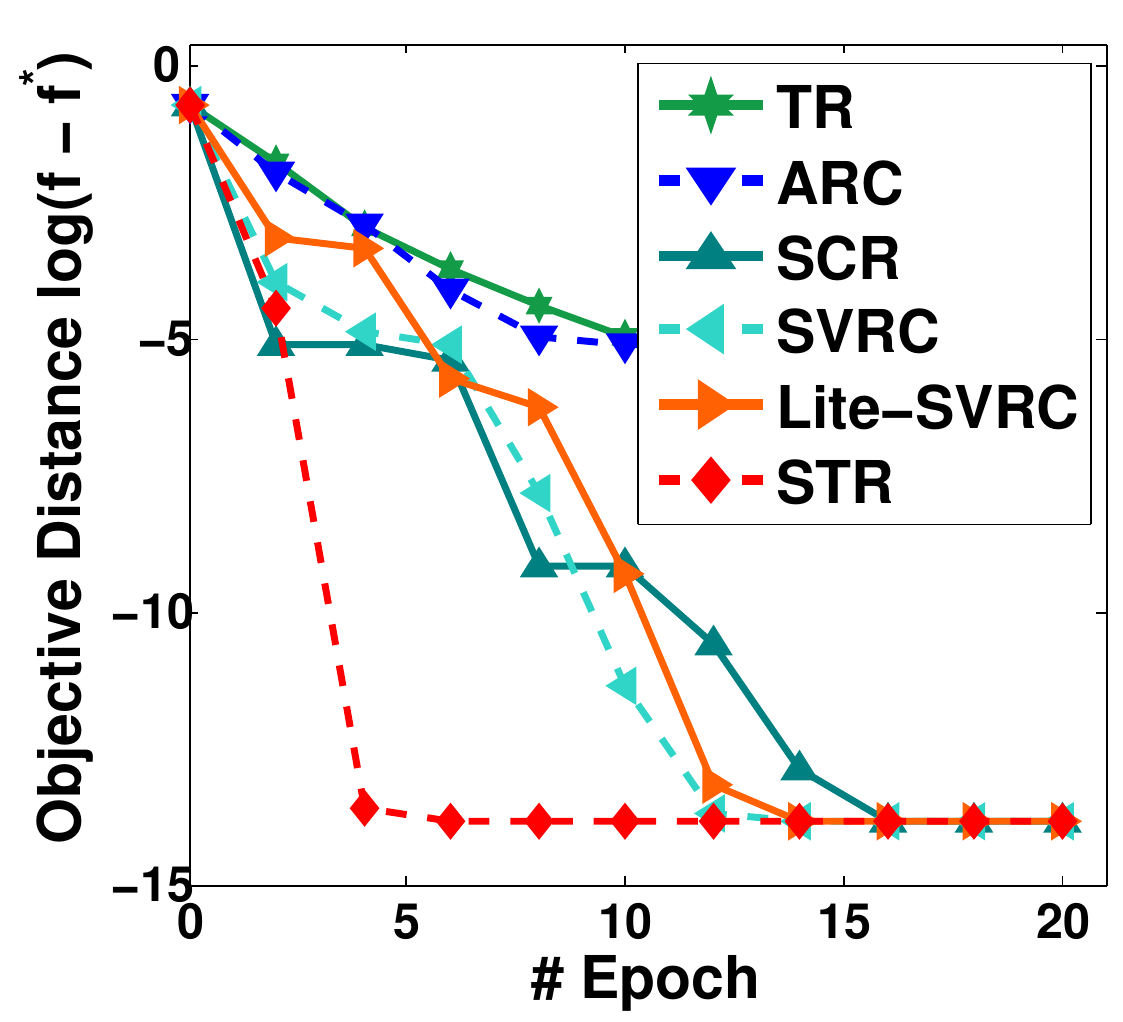}\\
			\multicolumn{2}{c}{{\small{(a) \textsf{a9a}}}}& \multicolumn{2}{c}{{\small{ (b) \textsf{ijcnn}}}}\\
			\includegraphics[width=0.245\linewidth]{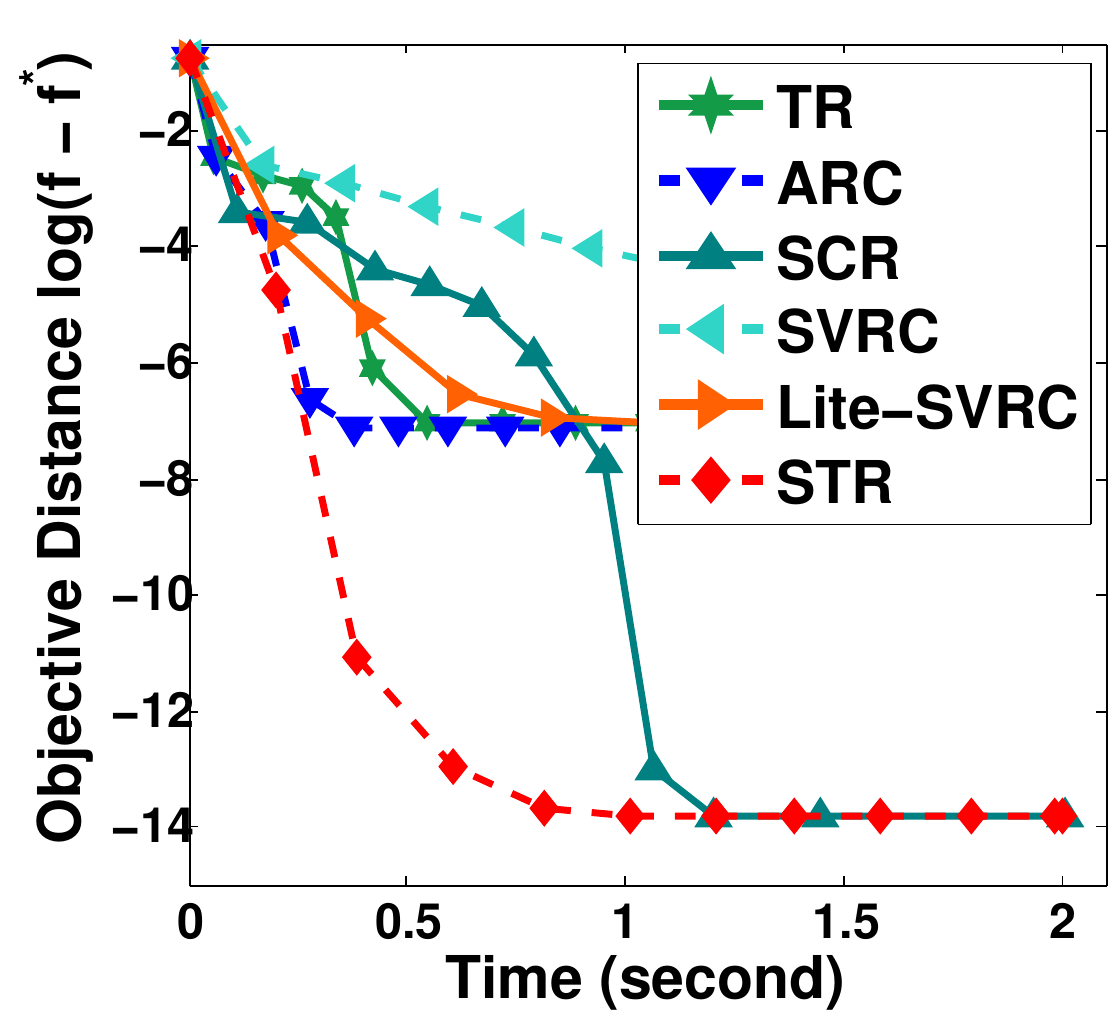}&
			\includegraphics[width=0.245\linewidth]{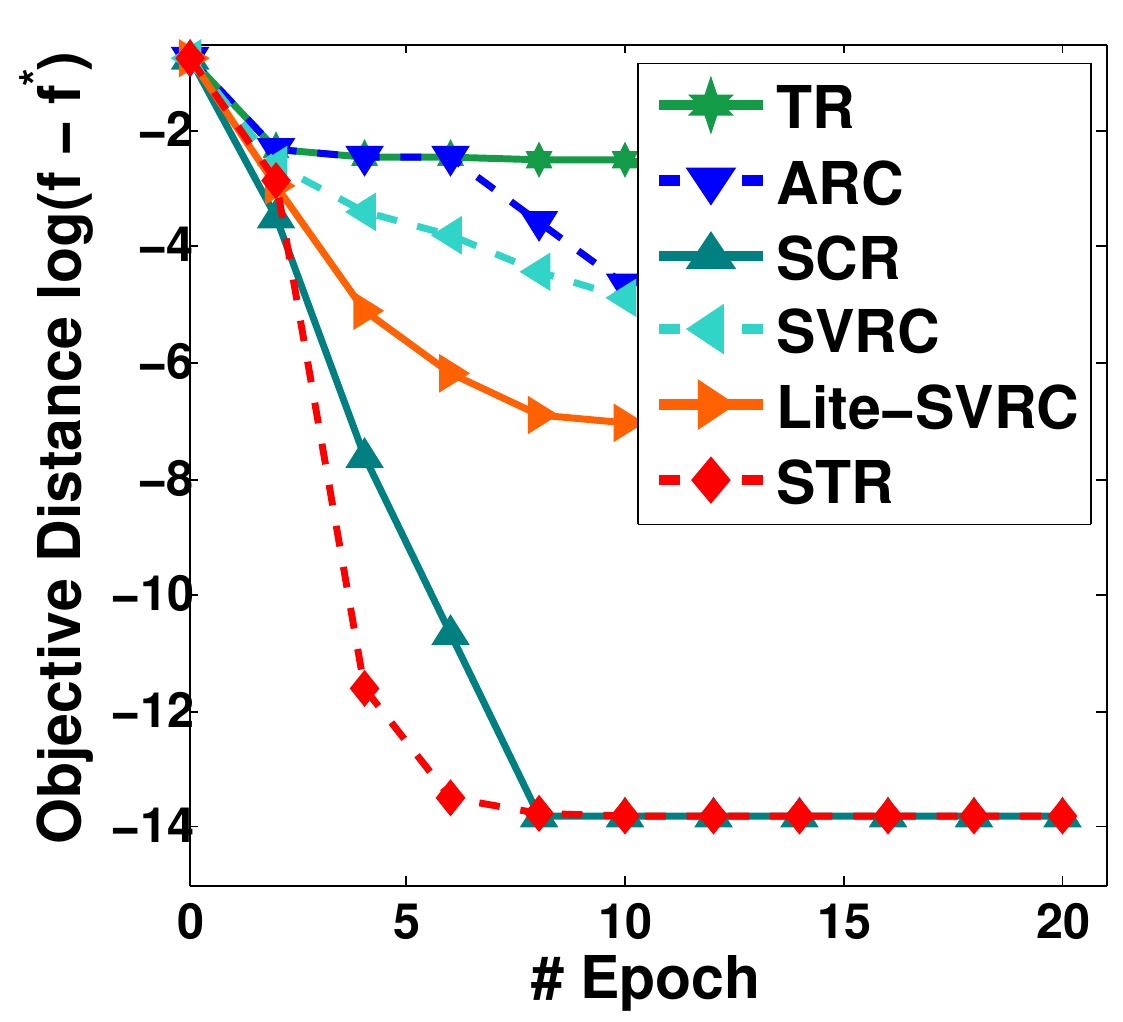}&
			\includegraphics[width=0.245\linewidth]{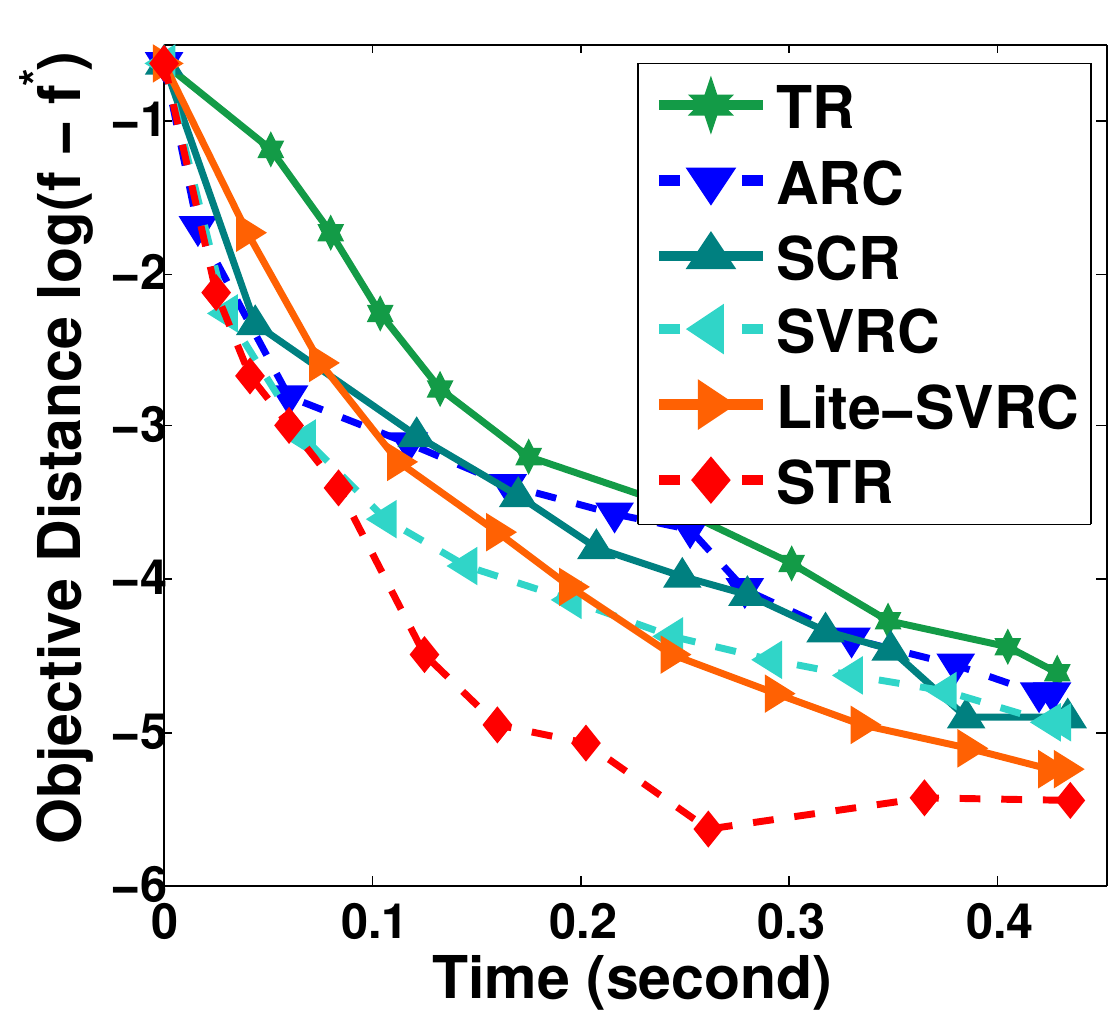}&
			\includegraphics[width=0.245\linewidth]{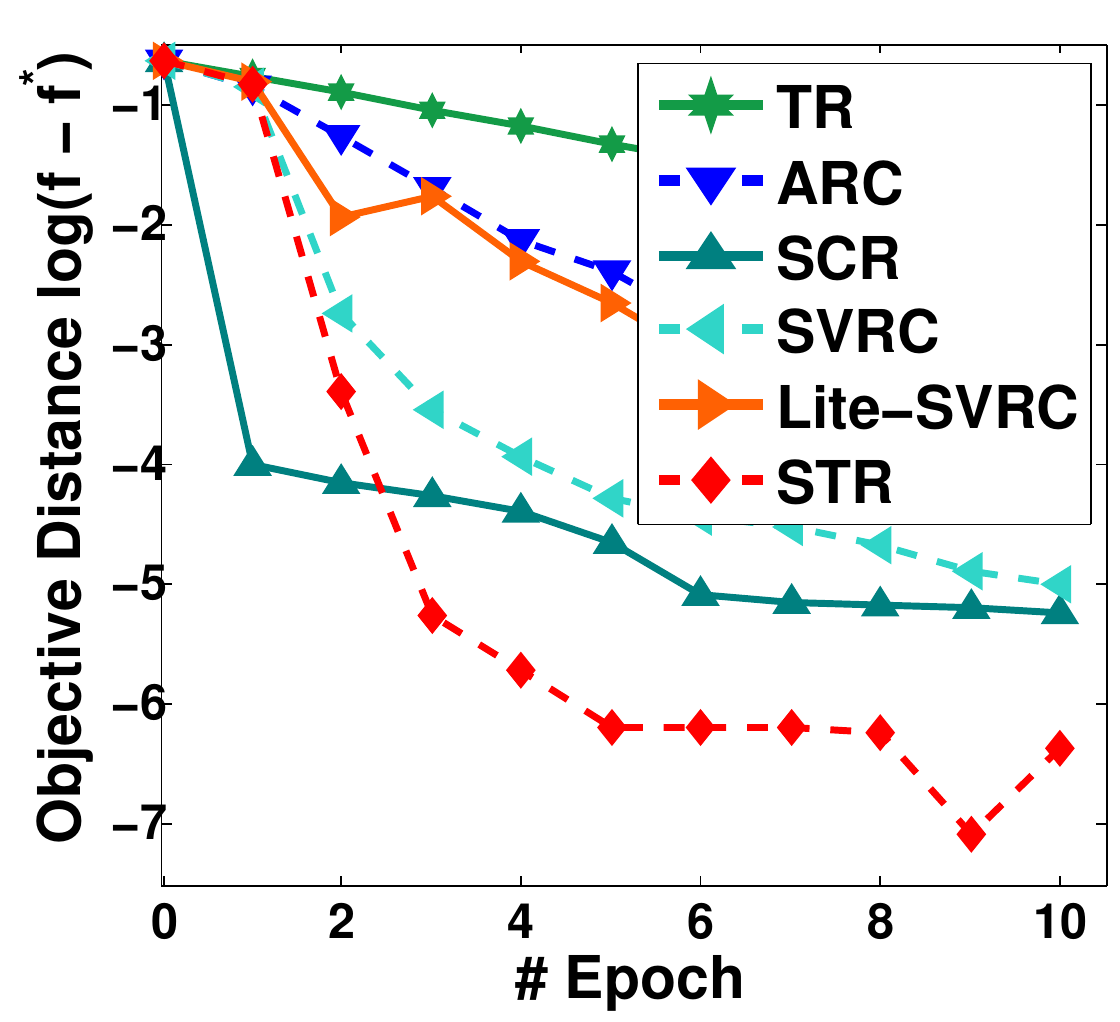}\\
			\multicolumn{2}{c}{{\small{(c) \textsf{codrna}}}}& \multicolumn{2}{c}{{\small{ (d) \textsf{phishing}}}}\\
		\end{tabular}
	\end{center}
	\vspace{-1.2em}
	\caption{Comparison on  the logistic regression with nonconvex regularizer. }
	\label{comparisonlogistic}
	\vspace{-1.0em}
\end{figure*}
\begin{figure*}[t]
	\begin{center}
		\setlength{\tabcolsep}{0.8pt} 
		\begin{tabular}{cccc}
			\includegraphics[width=0.245\linewidth]{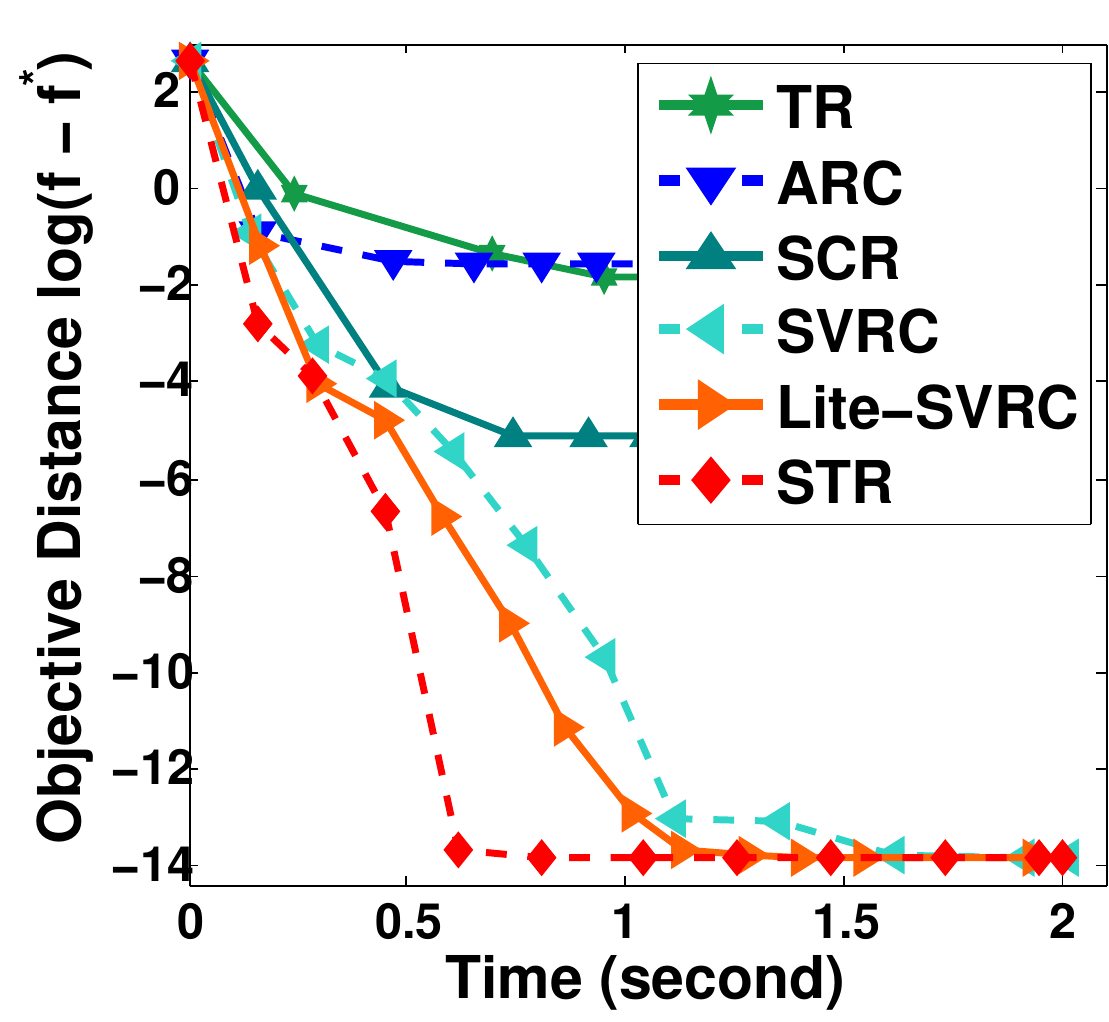}&
			\includegraphics[width=0.245\linewidth]{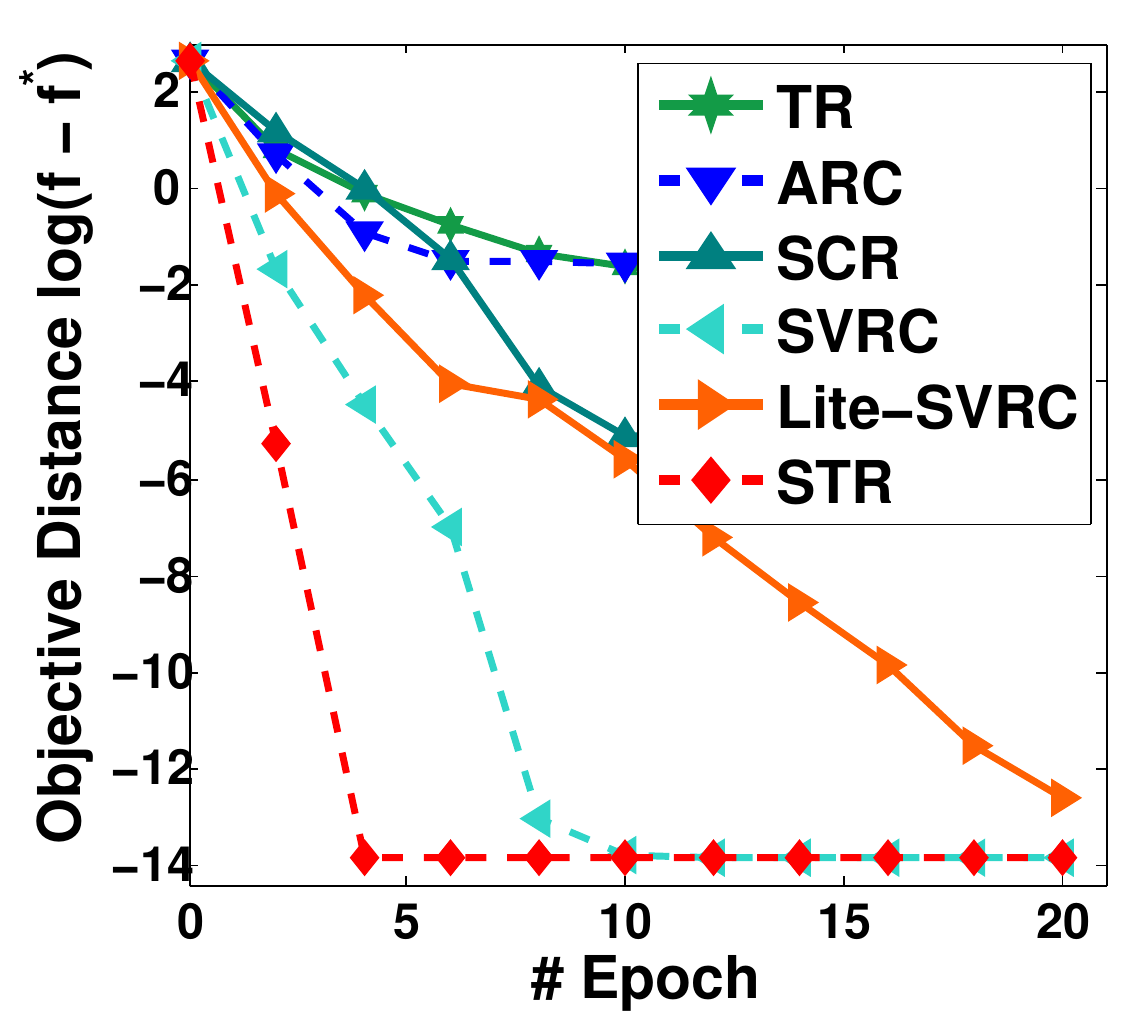}&
			\includegraphics[width=0.245\linewidth]{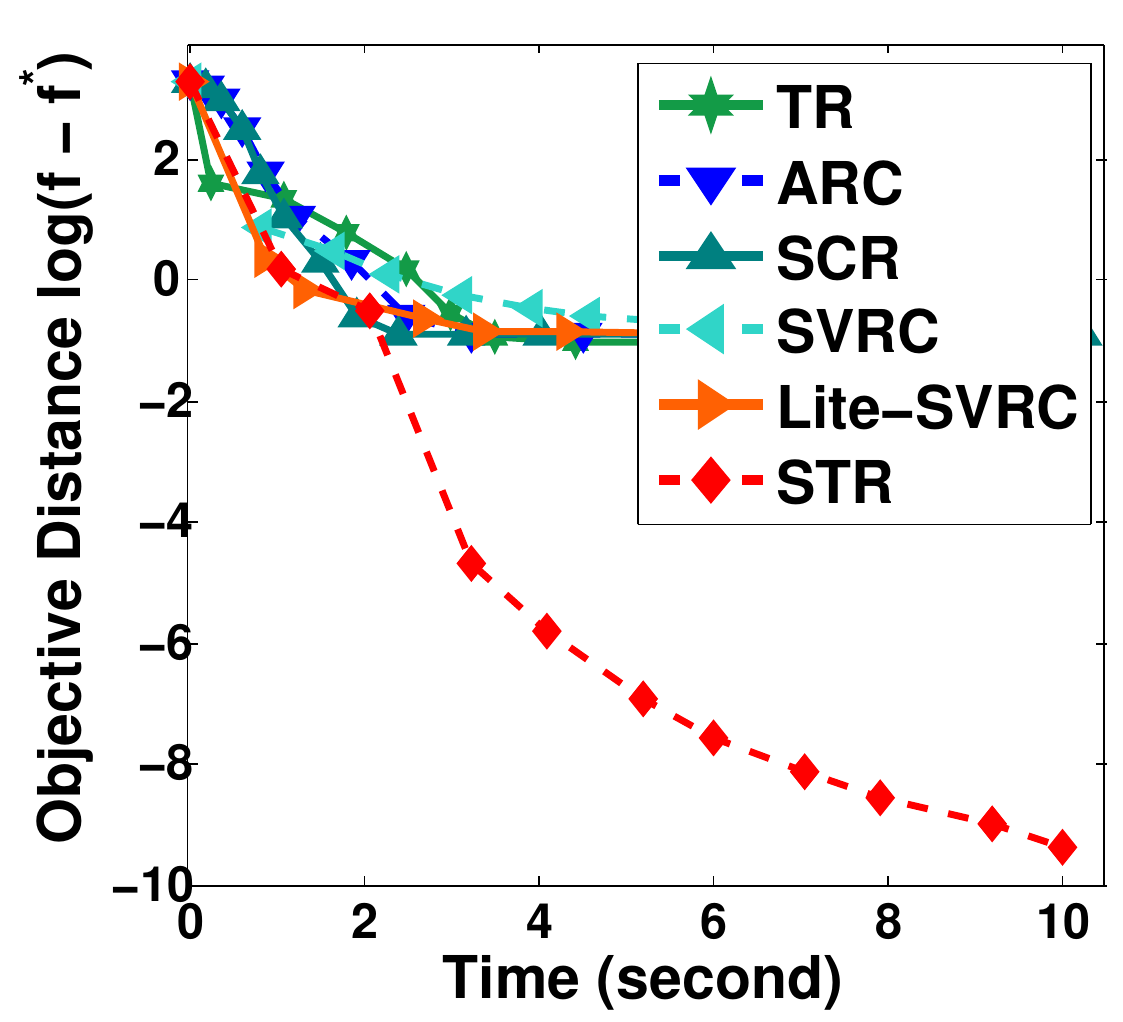}&
			\includegraphics[width=0.245\linewidth]{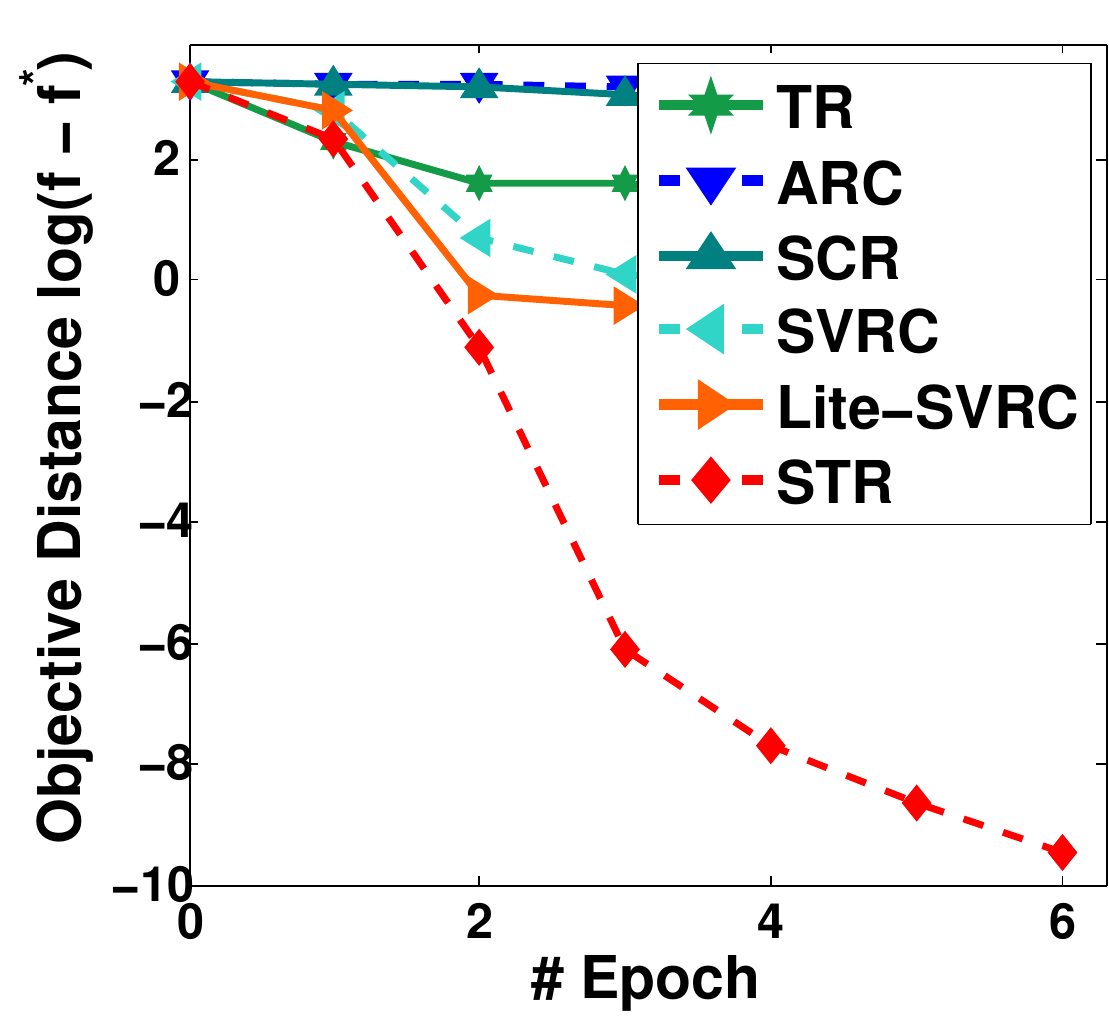}\\
			\multicolumn{2}{c}{{\small{(a) \textsf{a9a}}}}& \multicolumn{2}{c}{{\small{ (b) \textsf{w8a}}}}\\
			\includegraphics[width=0.245\linewidth]{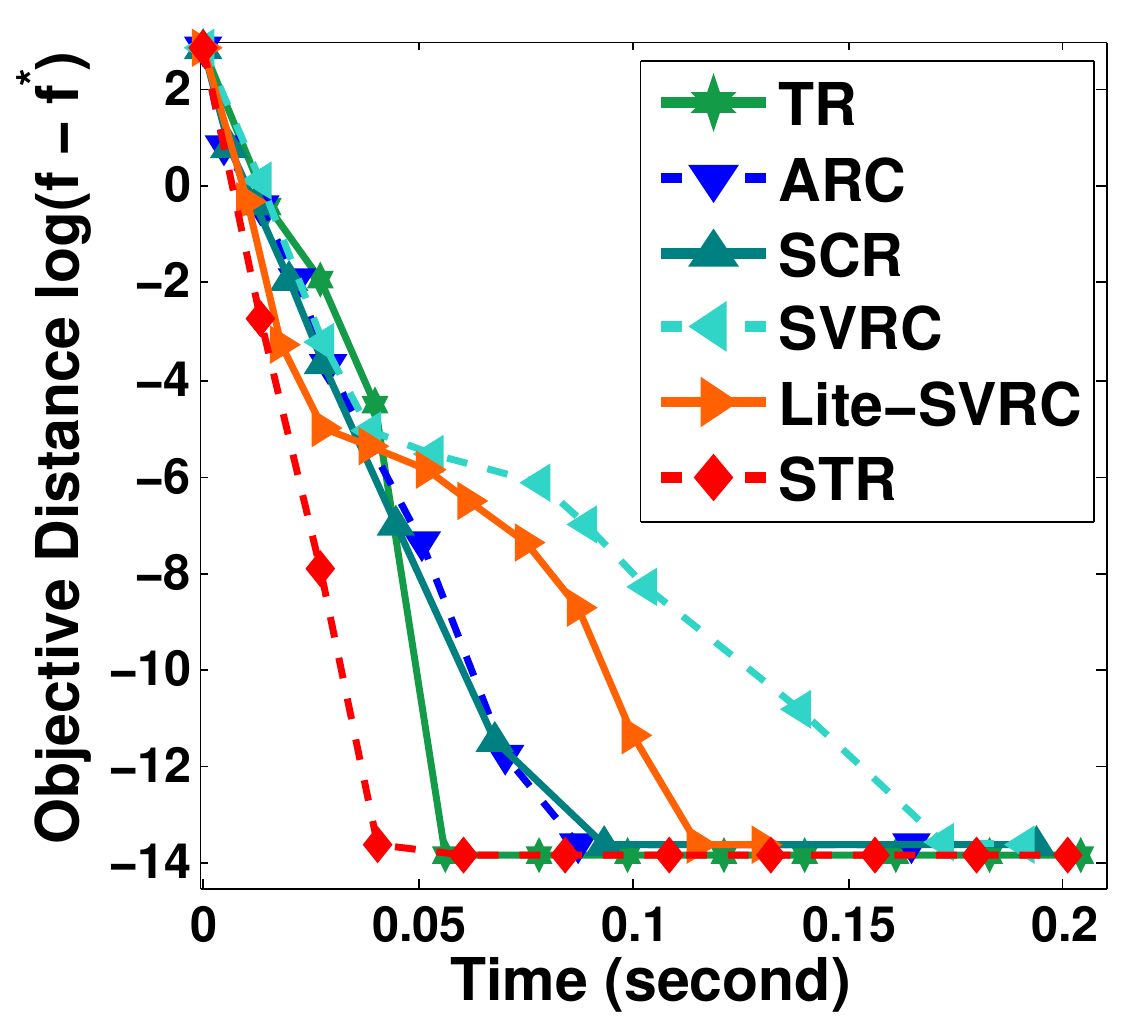}&
			\includegraphics[width=0.245\linewidth]{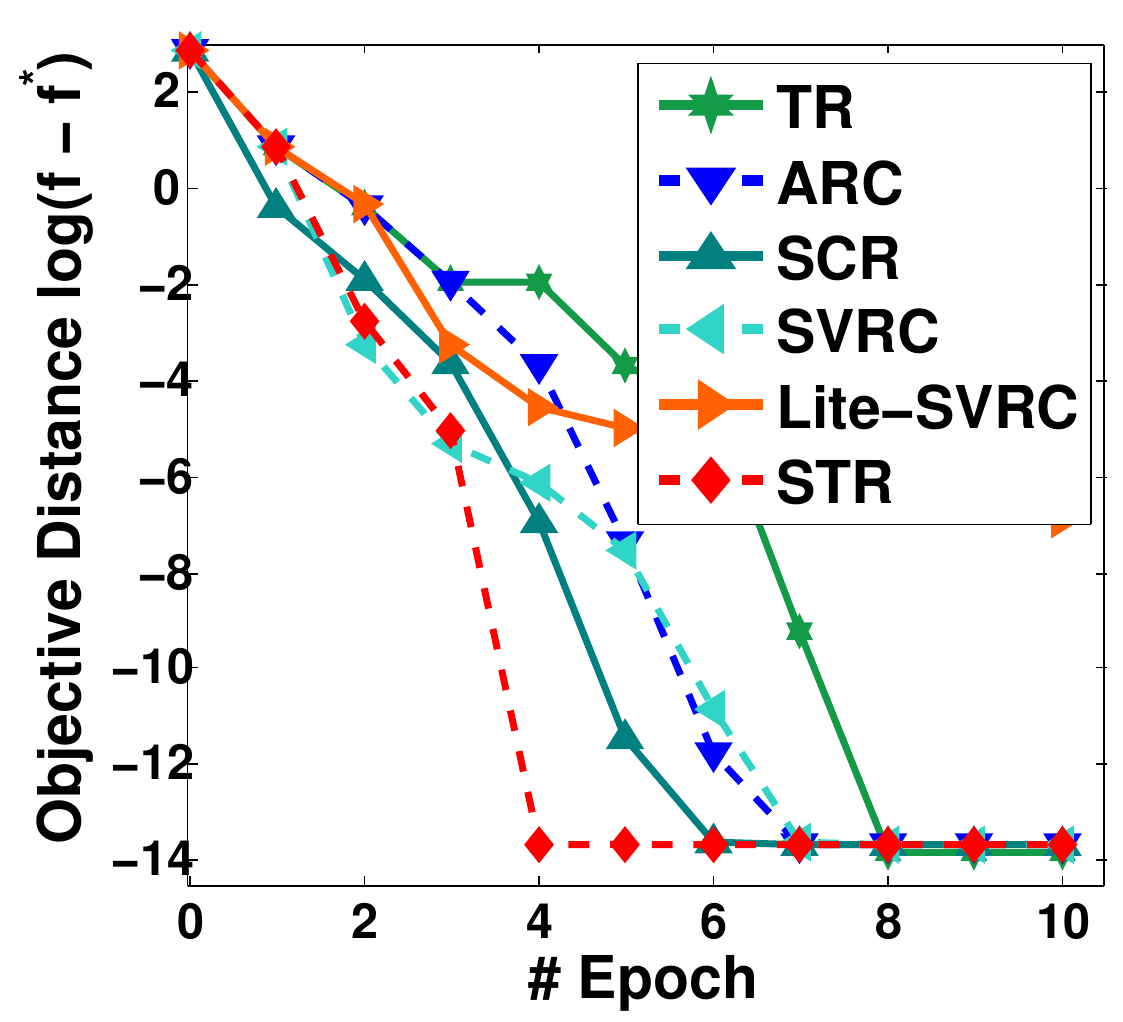}&
			\includegraphics[width=0.243\linewidth]{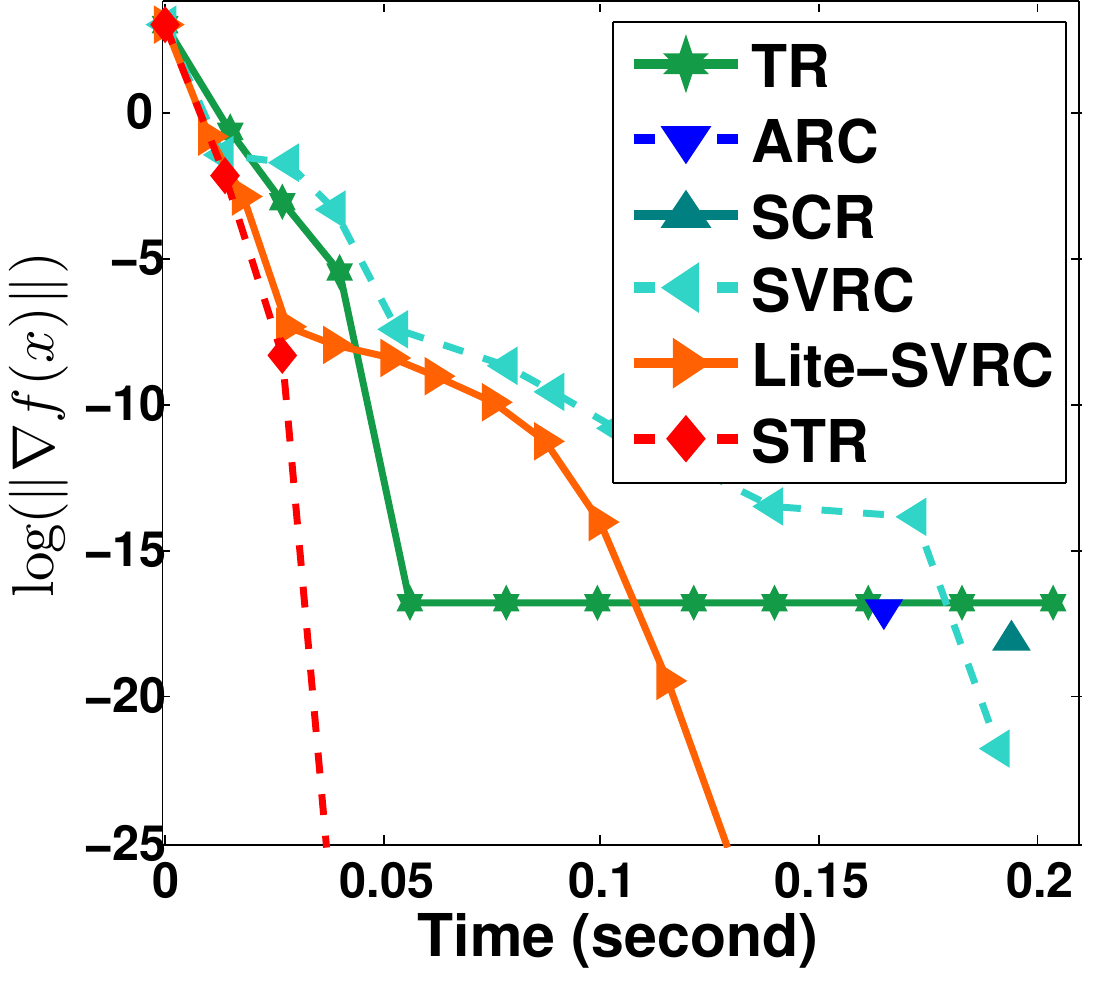}&
			\includegraphics[width=0.243\linewidth]{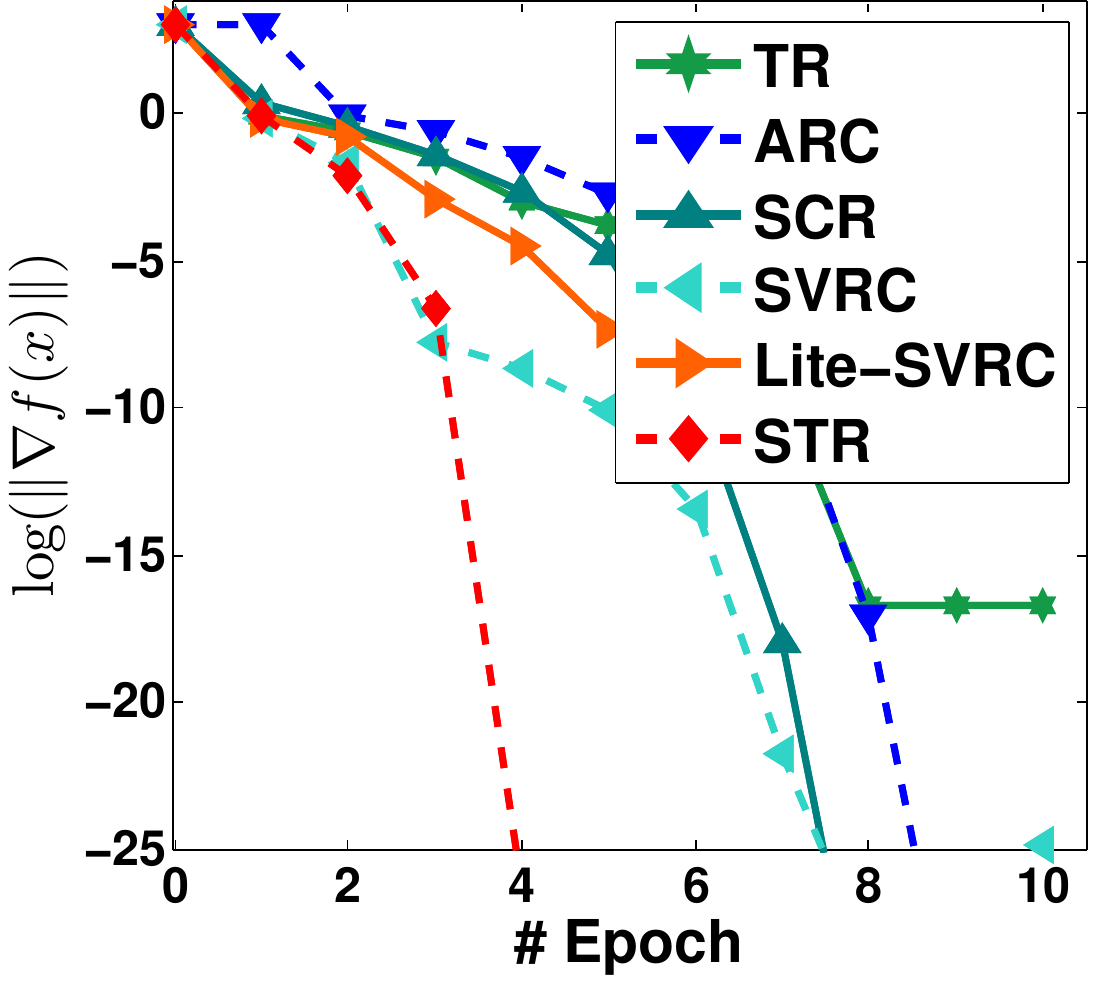}\\
			\multicolumn{4}{c}{{\small{(c) \textsf{codrna}}}}\\
		\end{tabular}
	\end{center}
	\vspace{-1.2em}
	\caption{Comparison   on  the  nonlinear least
		square problem.}
	\label{comparisonsquare}
	\vspace{-1.0em}
\end{figure*}

\subsection{Gradient Estimator: Case (2)}
When the maximum of stochastic gradient and Hessian complexities, is prioritized, we use Hessian to improve the gradient estimation in Algorithm \ref{alg_SPIDER_estimator}.
Intuitively, the Mean Value Theorem gives
\begin{equation*}
\nabla f_i(\xB^t) - \nabla f_i(\xB^{t-1}) = \nabla^2 f_i(\bar{\xB}^t),
\end{equation*}
with $\bar{\xB}^t = \alpha^t\cdot \xB^t + (1-\alpha^t)\cdot\xB^{t-1}$ for some $\alpha^t\in[0, 1]$,
which under Assumption \ref{ass_second_order_smooth} allow us to bound
\begin{equation*}
\|\nabla f_i(\xB^t) - \nabla f_i(\tilde{\xB}) - \nabla^2 f_i(\tilde{\xB})(\xB^t - \tilde{\xB})\| = L_2\|\xB^t - \tilde{\xB}\|^2.
\end{equation*}
Such property can be used to improve Lemma \ref{lemma_SPIDER} of Estimator \ref{alg_SPIDER_estimator}.
Specifically, define the correction term
\begin{equation*}
\cB^k = [\nabla^2 F(\tilde{\xB}) -\nabla^2 f(\tilde{\xB}; \GM)](\xB^k - \xB^{k-1}),
\end{equation*}
where $\tilde{\xB}$ is some reference point updated in a epoch-wise manner.
Estimator \ref{alg_CASPIDER_estimator} adds $\cB^k$ to the estimator in Estimator \ref{alg_SPIDER_estimator}.
Note that in Estimator \ref{alg_CASPIDER_estimator}, the number of first-order and second-order oracle complexities are the same.

We now analyze the necessary first-order (and second-order) oracle complexity to meet requirement (\ref{conditiongradient}).
\begin{lemma} \label{lemma_CASPIDER}
	Assume Algorithm \ref{alg_str_2} takes the trust region size $r = \sqrt{\epsilon/L_2}$ as in Theorem \ref{thm_TR_meta}.
	For any $k\geq 0$, Estimator \ref{alg_CASPIDER_estimator} produces estimator $\gB^{k}$ for the first order differential $\nabla F(\xB^{k})$ such that $\|\gB^{k} - \nabla F(\xB^{k})\|\leq \epsilon/6$ with probability at least $1-\delta/K_0$, if we set $p_1 = n^{0.25}$ and $s_1 = n^{0.75}c\log({K_0}/{\delta})$, with $c = 1152$.
	Consequently the amortized per-iteration stochastic first-order oracle complexity to construct $\gB^{k}$ is $2s_1 = 2n^{0.75}c\log{K_0}/{\delta}$.
\end{lemma}
The proof for Lemma \ref{lemma_CASPIDER} is similar to the one of Lemma \ref{lemma_Hessian_estimator} and is deferred to Appendix \ref{section_proof_lemma_caspider}.

Similar to the previous section, Lemma \ref{lemma_CASPIDER} only guarantee the gradient estimator satisfies the requirement \ref{conditiongradient} in a single iteration.
We extended such result to hold for all $k$ by using the union bound with $K_0 = 2K$, which together with Theorem \ref{thm_TR_meta} gives the following corollary.
\begin{corollary}
	Assume Algorithm \ref{alg_str_2} will use Estimator \ref{alg_CASPIDER_estimator} to construct the first-order differential estimator $\gB^k$ and use Estimator \ref{alg_HB_estimator} to construct the second-order differential estimator $\HB^k$.
	To find an $12\epsilon$-SOSP with probability at least $1-\delta$, the overall stochastic first-order oracle complexity is ${20000n^{0.75}\sqrt{L_2}\Delta}/{\epsilon^{1.5}}$ and the overall stochastic second-order oracle complexity is $c_1{n^{0.75}\sqrt{L_2}\Delta}/{\epsilon^{1.5}}$.
	\label{corollary_str_2}
\end{corollary}
Corollary \ref{corollary_str_2} shows that to find an $\epsilon$-SOSP for Problem \ref{eqn_stochastic_problem}, both stochastic first-order and second-order oracle complexities of STR$_2$ are $\OM(n^{3/4}/\epsilon^{1.5})$ which is better than the best existing result $\OM(n^{4/5}/\epsilon^{1.5})$ in \cite{zhou2018stochastic}.

%

\section{Experiments}\label{experiments}
In this section, we compare the proposed STR with several state-of-the-art (stochastic) cubic regularized algorithms and trust region approaches, including trust region (TR) algorithm~\cite{conn2000trust}, adaptive cubic regularization (ARC)~\cite{cartis2011adaptive}, sub-sampled cubic regularization (SCR)~\cite{kohler2017sub}, stochastic variance-reduced cubic (SVRC)~\cite{zhou2018stochastic} and Lite-SVRC~\cite{zhou2018sample}. For STR, we estimate the gradient as the way in case (1). This is because such a method enjoys lower Hessian computational complexity over the way in case (2) and for most problems, computing their Hessian matrices is much more time-consuming than computing their gradients. For the subproblems in these compared methods, we use Lanczos method~\cite{gould1999solving,kohler2017sub} to solve the sub-problem approximately in a Hessian-related Krylov subspace.
We run simulations on five datasets from LibSVM (\textsf{a09}, \textsf{ijcnn}, \textsf{codrna}, \textsf{phishing}, and \textsf{w08}). The details of these datasets are described in Appendix~\ref{append:more_experiment}. For all the considered algorithms,
we tune their hyper-parameters optimally.

\textbf{Two evaluation nonconvex problems.} Following~\cite{kohler2017sub,zhou2018stochastic}, we evaluate all considered algorithms on two learning tasks: the logistic regression with nonconvex regularizer and the nonlinear least square. Given $n$ data points $(\bm{x}_i,y_i)$ where $\bm{x}_i\in \mathbb{R}^d$ is the sample vector and $y_i\in\{-1,1\}$ is the label, \emph{logistic regression with nonconvex regularizer} aims at distinguishing these two kinds of samples by solving the following problem
\begin{equation*}
\min_{\bm{w}} \frac{1}{n}\sum_{i=1}^{n} \log(1+\exp(-y_i \bm{w}^T\bm{x}_i)) + \lambda R(\bm{w};\alpha),
\end{equation*}
where the nonconvex regularizer $R(\bm{w};\alpha)$ is defined as $ R(\bm{w};\alpha)= \sum_{i=1}^{d} {\alpha \bm{w}_i^2}/({1+\alpha \bm{w}_i^2})$.  The \emph{nonlinear least square} problem fits the nonlinear data by minimizing
\begin{eqnarray*}
\begin{split}
\min_{\bm{w}} \frac{1}{2n}\sum_{i=1}^{n} \left[ y_i -\phi(\bm{w}^T\bm{x}_i)\right]^2 + \lambda R(\bm{w},\alpha).
\end{split}
\end{eqnarray*}
For both these two kinds of problems, we set the parameters $\lambda=10^{-3}$ and $\alpha=10$ for all testing datasets.

Figure~\ref{comparisonlogistic} summarizes the testing results on the nonconvex logistic regression problems. For each dataset, we report the function value gap v.s. the overall algorithm running time which can reflect the overall computational complexity of an algorithm,  and also show the function value gap v.s. Hessian sample complexity which reveals the  complexity of Hessian computation. From Fig.~\ref{comparisonlogistic}, one can observe that our proposed STR algorithm  runs faster than the compared algorithms in terms of the algorithm running time, showing the overall superiority of STR. Furthermore, STR also reveals much sharper convergence curves in terms of the Hessian sample complexity which is consistent with our theory. This is because to achieve an $\epsilon$-accuracy local minimum,  the Hessian sample complexity of the proposed STR is $\mathcal{O}(n^{0.5}/\epsilon^{1.5})$ and is superior over the complexity of the compared methods (see the comparison in Sec.~\ref{resultsofcomplexity}).
Indeed, this also explains why our algorithm is also faster in terms of algorithm running time, since for most optimization problems, Hessian matrix is much more computationally expensive than the gradient and thus  more efficient Hessian sample complexity means faster overall convergence speed.

Figure~\ref{comparisonsquare} displays the results of the compared algorithms on the nonlinear least square problems.  STR shows very similar behaviors as those in Figure~\ref{comparisonlogistic}. More specifically, STR achieves fastest convergence rate in terms of both algorithm running time and Hessian sample complexity. On the \textsf{codrna} dataset (the bottom of Figure~\ref{comparisonsquare}) we further plot the function value gap versus running-time curves and Hessian sample complexity. One can obverse that the gradient in STR vanishes significantly faster than other algorithms which means that STR can find the stationary point with high efficiency.    See Figure~\ref{comparisongradient} in Appendix~\ref{appedxexp} for more experimental results on running time comparison.  All these results confirm the superiority  of the proposed STR.

\section*{Conclusion}
We proposed two stochastic trust region variants.
Under two efficiency measurement settings (whether the stochastic first- and second-order oracle complexity are treated equally), the proposed method achieve state-of-the-art oracle complexity.
Experimental results well testify our theoretical implications and the efficiency of the proposed algorithm.
\clearpage
\bibliographystyle{plainnat}
\bibliography{STRV}

\begin{thebibliography}{19}
\providecommand{\natexlab}[1]{#1}
\providecommand{\url}[1]{\texttt{#1}}
\expandafter\ifx\csname urlstyle\endcsname\relax
  \providecommand{\doi}[1]{doi: #1}\else
  \providecommand{\doi}{doi: \begingroup \urlstyle{rm}\Url}\fi

\bibitem[Carmon et~al.(2017)Carmon, Duchi, Hinder, and
  Sidford]{carmon2017lower}
Yair Carmon, John~C Duchi, Oliver Hinder, and Aaron Sidford.
\newblock Lower bounds for finding stationary points i.
\newblock \emph{arXiv preprint arXiv:1710.11606}, 2017.

\bibitem[Cartis et~al.(2011)Cartis, Gould, and Toint]{cartis2011adaptive}
Coralia Cartis, Nicholas~IM Gould, and Philippe~L Toint.
\newblock Adaptive cubic regularisation methods for unconstrained optimization.
  part i: motivation, convergence and numerical results.
\newblock \emph{Mathematical Programming}, 127\penalty0 (2):\penalty0 245--295,
  2011.

\bibitem[Conn et~al.(2000)Conn, Gould, and Toint]{conn2000trust}
Andrew~R Conn, Nicholas~IM Gould, and Ph~L Toint.
\newblock \emph{Trust region methods}, volume~1.
\newblock Siam, 2000.

\bibitem[Curtis et~al.(2017)Curtis, Robinson, and Samadi]{curtis2017trust}
Frank~E Curtis, Daniel~P Robinson, and Mohammadreza Samadi.
\newblock A trust region algorithm with a worst-case iteration complexity of
  (epsilon 3/2) for nonconvex optimization.
\newblock \emph{Mathematical Programming: Series A and B}, 162\penalty0
  (1-2):\penalty0 1--32, 2017.

\bibitem[Dauphin et~al.(2014)Dauphin, Pascanu, Gulcehre, Cho, Ganguli, and
  Bengio]{dauphin2014identifying}
Yann~N Dauphin, Razvan Pascanu, Caglar Gulcehre, Kyunghyun Cho, Surya Ganguli,
  and Yoshua Bengio.
\newblock Identifying and attacking the saddle point problem in
  high-dimensional non-convex optimization.
\newblock In \emph{Advances in neural information processing systems}, pages
  2933--2941, 2014.

\bibitem[Fang et~al.(2018)Fang, Li, Lin, and Zhang]{fang2018spider}
Cong Fang, Chris~Junchi Li, Zhouchen Lin, and Tong Zhang.
\newblock Spider: Near-optimal non-convex optimization via stochastic
  path-integrated differential estimator.
\newblock In \emph{Advances in Neural Information Processing Systems}, pages
  686--696, 2018.

\bibitem[Gould et~al.(1999)Gould, Lucidi, Roma, and Toint]{gould1999solving}
Nicholas~IM Gould, Stefano Lucidi, Massimo Roma, and Philippe~L Toint.
\newblock Solving the trust-region subproblem using the lanczos method.
\newblock \emph{SIAM Journal on Optimization}, 9\penalty0 (2):\penalty0
  504--525, 1999.

\bibitem[Gower et~al.(2018)Gower, Roux, and Bach]{pmlr-v84-gower18a}
Robert Gower, Nicolas~Le Roux, and Francis Bach.
\newblock Tracking the gradients using the hessian: A new look at variance
  reducing stochastic methods.
\newblock In \emph{Proceedings of the Twenty-First International Conference on
  Artificial Intelligence and Statistics}, 2018.

\bibitem[Johnson and Zhang(2013)]{johnson2013accelerating}
Rie Johnson and Tong Zhang.
\newblock Accelerating stochastic gradient descent using predictive variance
  reduction.
\newblock In \emph{Advances in neural information processing systems}, pages
  315--323, 2013.

\bibitem[Kohler and Lucchi(2017{\natexlab{a}})]{kohler2017sub}
Jonas~Moritz Kohler and Aurelien Lucchi.
\newblock Sub-sampled cubic regularization for non-convex optimization.
\newblock \emph{arXiv preprint arXiv:1705.05933}, 2017{\natexlab{a}}.

\bibitem[Kohler and Lucchi(2017{\natexlab{b}})]{pmlr-v70-kohler17a}
Jonas~Moritz Kohler and Aurelien Lucchi.
\newblock Sub-sampled cubic regularization for non-convex optimization.
\newblock In \emph{Proceedings of the 34th International Conference on Machine
  Learning}, 2017{\natexlab{b}}.

\bibitem[Nesterov and Polyak(2006)]{nesterov2006cubic}
Yurii Nesterov and Boris~T Polyak.
\newblock Cubic regularization of newton method and its global performance.
\newblock \emph{Mathematical Programming}, 108\penalty0 (1):\penalty0 177--205,
  2006.

\bibitem[Nguyen et~al.(2017)Nguyen, Liu, Scheinberg, and
  Tak{\'a}{\v{c}}]{nguyen2017sarah}
Lam~M Nguyen, Jie Liu, Katya Scheinberg, and Martin Tak{\'a}{\v{c}}.
\newblock Sarah: A novel method for machine learning problems using stochastic
  recursive gradient.
\newblock In \emph{International Conference on Machine Learning}, pages
  2613--2621, 2017.

\bibitem[Tropp(2012)]{tropp2012user}
Joel~A Tropp.
\newblock User-friendly tail bounds for sums of random matrices.
\newblock \emph{Foundations of computational mathematics}, 12\penalty0
  (4):\penalty0 389--434, 2012.

\bibitem[Xu et~al.(2017)Xu, Roosta-Khorasani, and Mahoney]{xu2017newton}
Peng Xu, Farbod Roosta-Khorasani, and Michael~W Mahoney.
\newblock Newton-type methods for non-convex optimization under inexact hessian
  information.
\newblock \emph{arXiv preprint arXiv:1708.07164}, 2017.

\bibitem[Zhang et~al.(2018)Zhang, Xiao, and Zhang]{zhang2018adaptive}
Junyu Zhang, Lin Xiao, and Shuzhong Zhang.
\newblock Adaptive stochastic variance reduction for subsampled newton method
  with cubic regularization.
\newblock \emph{arXiv preprint arXiv:1811.11637}, 2018.

\bibitem[Zhou et~al.(2018{\natexlab{a}})Zhou, Xu, and Gu]{pmlr-v80-zhou18d}
Dongruo Zhou, Pan Xu, and Quanquan Gu.
\newblock Stochastic variance-reduced cubic regularized {N}ewton methods.
\newblock In \emph{Proceedings of the 35th International Conference on Machine
  Learning}, 2018{\natexlab{a}}.

\bibitem[Zhou et~al.(2018{\natexlab{b}})Zhou, Xu, and Gu]{zhou2018sample}
Dongruo Zhou, Pan Xu, and Quanquan Gu.
\newblock Sample efficient stochastic variance-reduced cubic regularization
  method.
\newblock \emph{arXiv preprint arXiv:1811.11989}, 2018{\natexlab{b}}.

\bibitem[Zhou et~al.(2018{\natexlab{c}})Zhou, Xu, and Gu]{zhou2018stochastic}
Dongruo Zhou, Pan Xu, and Quanquan Gu.
\newblock Stochastic variance-reduced cubic regularized newton method.
\newblock In \emph{ICML}, 2018{\natexlab{c}}.

\end{thebibliography}
\clearpage
\section{Appendix}
\subsection{Proof of Lemma \ref{lemma_SPIDER}} \label{section_proof_lemma_spider}
	Without loss of generality, we analyze the case $0\leq k< q_1$ for ease of notation.
	Define for $k\geq 1$ and $i\in[s_1]$ $$\aB_i^k \defi \nabla f_i(\xB^{k}) - \nabla f_i(\xB^{k-1}) - (\nabla F(\xB^{k}) - \nabla F(\xB^{k-1})).$$
	$\aB_i^k$ is a martingale difference: for all $k$ and $i$
	\begin{equation*}
	\EBB[\aB_i^k|\xB^k] = 0.
	\end{equation*}
	Besides $\aB_i^k$ has bounded norm: Using Assumption \ref{ass_second_order_smooth}
	\begin{align}
	\|\aB_i^k\|\! \leq& \|\nabla f_i(\xB^{k}) \!-\! \nabla f_i(\xB^{k-1})\| \!+\! \|\nabla F(\xB^{k}) \!-\! \nabla F(\xB^{k-1})\| \notag\\
	\leq &L_1\|\xB^{k}-\xB^{k-1}\|\!+\! L_1\|\xB^{k} -\xB^{k-1}\| \notag\\
	\leq& 2L_1\sqrt{\epsilon/ L_2}. 	\label{eqn_SPIDER_proof}
	\end{align}
	From the construction of $\gB^{k}$, we have $$\gB^k -\nabla F(\xB^{k})= \sum_{j=1}^{k}\sum_{i=1}^{s_1} \frac{\aB_i^{j}}{s_1}.$$
	Recall the Azuma's Inequality.
	Using $k\leq p_1$, we have
	\begin{equation*}
	\begin{aligned}
	&Pr\{\|\gB^k -\nabla F(\xB^{k})\| \geq t\} \\
	\leq &\exp\{ -\frac{t^2/8}{\sum_{j=1}^{k}\sum_{i=1}^{s_1}\frac{4\epsilon L_1^2}{L_2s_1^2}}\}\leq \exp\{ -\frac{t^2/8}{{4\epsilon L_1^2 p_1/(s_1L_2)}}\}.
	\end{aligned}
	\end{equation*}
	Take $t=\epsilon/6$ and denote $c = 1152$.
	To ensure that
	\begin{equation*}
	Pr\{\|\gB^k -\nabla F(\xB^{k})\| \geq \epsilon/6\} \leq \delta/K_0,
	\end{equation*}
	we need $\frac{cL_1^2}{L_2}\log\frac{K_0}{\delta} \leq \frac{\epsilon s_1}{p_1}$.
	The best amortized stochastic first-order oracle complexity can be obtain by solving the following two-dimensional programming:
	\begin{equation*}
	\begin{aligned}
	\min_{p_1\geq 1, s_1\geq 1} \quad &(n + s_1(p_1-1))/p_1 \\
	s.t. \quad & \frac{cL_1^2}{L_2}\log\frac{K_0}{\delta} \leq \frac{\epsilon s_1}{p_1},
	\end{aligned}
	\end{equation*}
	which has the solution $s_1 = \min\{n, \sqrt{\frac{n}{\epsilon}\cdot\frac{cL_1^2 \log\frac{K_0}{\delta}}{L_2}}\}$, and $p_1 = \max\{1, \sqrt{n\epsilon\cdot \frac{L_2}{cL_1^2 \log\frac{K_0}{\delta}}}\}$.
	Note that when we take $s_1 = n$, we directly compute $\gB^k = \nabla F(\xB^k)$ without sampling.
	
	The amortized stochastic first-order oracle complexity is obtain by plugging in the choice of $s_1$ and $p_1$.
\subsection{Proof of Lemma \ref{lemma_CASPIDER}} \label{section_proof_lemma_caspider}
	Without loss of generality, we analyze the case $0\leq k< q_1$ for ease of notation.
	Define for $k\geq 1$ and $i\in[s_1]$
	\begin{equation*}
	\begin{aligned}
	\bB_i^k \defi& \nabla f_i(\xB^{k}) - \nabla f_i(\xB^{k-1}) - \nabla^2 f_i(\tilde{\xB})(\xB^k - \xB^{k-1})\\
	&- [\nabla F(\xB^{k}) - \nabla F(\xB^{k-1}) - \nabla^2 F(\tilde{\xB})(\xB^k - \xB^{k-1})].
	\end{aligned}	
	\end{equation*}
	$\bB_i^k$ is a martingale difference: for all $k$ and $i$
	\begin{equation*}
	\EBB[\bB_i^k|\xB^k] = 0.
	\end{equation*}
	Besides $\bB_i^k$ has bounded norm:
	\begin{align*}
	\|\bB_i^k\|\! \leq& \|\nabla f_i(\xB^{k}) \!-\! \nabla f_i(\xB^{k-1}) - \nabla^2 f_i(\tilde{\xB})(\xB^k - \xB^{k-1})\| \\
	&\!+\! \|\nabla F(\xB^{k}) \!-\! \nabla F(\xB^{k-1}) - \nabla^2 F(\tilde{\xB})(\xB^k - \xB^{k-1})\| \notag\\
	= &\|[\nabla^2 f_i(\bar{\xB})-\nabla^2 f_i(\tilde{\xB})](\xB^{k}-\xB^{k-1})\|\\
	&\!+\! \|[\nabla^2 F(\bar{\xB}')- \nabla^2 F(\tilde{\xB})](\xB^{k} -\xB^{k-1})\| \notag\\
	\end{align*}
	where $\bar{\xB}$ and $\bar{\xB}'$ are two points obtained from mean value theorem.
	Using $\|\bar{\xB} - \tilde{\xB}\|\leq k\cdot r$ and $\|\bar{\xB}' - \tilde{\xB}\|\leq k\cdot r$, and Assumption \ref{ass_second_order_smooth}, where $r$ is the trust region size, we bound
	\begin{align*}
	\|\bB_i^k\|\leq& L_2\|\bar{\xB} - \tilde{\xB}\|\|\xB^{k}-\xB^{k-1}\|\!+\! L_2\|\bar{\xB}'-\tilde{\xB}\|\|\xB^{k}-\xB^{k-1}\| \notag\\
	\leq& 2L_2kr^2 \leq 2p_1\epsilon
	\end{align*}
	From the construction of $\gB^{k}$, we have $$\gB^k -\nabla F(\xB^{k})= \sum_{j=1}^{k}\sum_{i=1}^{s_1} \frac{\bB_i^{j}}{s_1}.$$
	We use $k\leq p_1$ and the Azuma's inequality to bound
	\begin{equation*}
	\begin{aligned}
	&Pr\{\|\gB^k -\nabla F(\xB^{k})\| \geq t\} \\
	\leq &\exp\{ -\frac{t^2/8}{\sum_{j=1}^{k}\sum_{i=1}^{s_1}\frac{4p_1^2\epsilon^2}{s_1^2}}\}\leq \exp\{ -\frac{t^2/8}{{4\epsilon^2 p_1^3/s_1}}\}.
	\end{aligned}
	\end{equation*}
	Thus, by taking $t = \epsilon/6$ and $c = 1152$, we need $\frac{s_1}{p_1^3}\geq c\log\frac{K_0}{\delta}$.
	Further we want $s_1p_1 \simeq \OM(n)$ and hence we take $p_1 = n^{0.25}$ and $s_1 = n^{0.75}c\log\frac{K_0}{\delta}$.
	The amortized stochastic first-order oracle complexity is bounded by $2s_1$.
\section{A Stochastic Trust Region Meta Algorithm}

\begin{table*}[h]
	\caption{Descriptions of the five testing datasets.}
	\setlength{\tabcolsep}{6.5pt} 
	\renewcommand{\arraystretch}{0.94}
	\label{Tabledatasets}
	\centering
	{ \footnotesize {
			\begin{tabular}{lcc|lcc}
				\toprule
				&  $\#$sample& $\#$feature & &  $\#$sample& $\#$feature\\  \midrule
				\textsf{a9a} &  32,561& 123 & \textsf{w8a} &49,749&300\\
				\textsf{ijcnn} & 49,990 &22 &\textsf{phishing}&   7,604&68\\
				\textsf{codrna} &  28,305 &8 &     \\
				\bottomrule
				\hline
			\end{tabular}
	}}
	\vspace{-0em}
\end{table*}
\begin{algorithm}[t]
	\floatname{algorithm}{MetaAlgorithm}
	\caption{Inexact Trust Region Method II}
	\label{alg_stochastic_trust_region_non_dual}
	\begin{algorithmic}[1]
		\REQUIRE Initialization $\xB^{0}$, step size $r$, number of iterations $K$, constructions of differential estimators $\gB^{k}$ and $\HB^{k}$
		\FOR{$k = 1$ {\bfseries to} $K$}
		\STATE Compute $\hB^{k}$ by solving \eqref{eqn_qcqp_inexact};
		\STATE $\xB^{k+1} := \xB^{k} + \hB^{k}$;
		\IF{$\|\hB^{k}\|<r$}
		\STATE {\bf Return $\xB_\epsilon := \xB^{k+1}$}
		\ENDIF
		\ENDFOR
		\STATE Randomly select $\bar{k}$ from $[K]$;
		\STATE {\bf Return } $\xB_\epsilon := \xB^{\bar k+1}$;
	\end{algorithmic}
\end{algorithm}
While we use the dual variable $\lambda^{k}$ in MetaAlgorithm \ref{alg_stochastic_trust_region} as a stopping criterion, we present MetaAlgorithm \ref{alg_stochastic_trust_region_non_dual} without using such quantity.
The following theorem shows that when the differential estimators satisfy condition \eqref{conditiongradient}, the similar $\OM(1/k^{2/3})$ convergence rate can be obtain in expectation.
\begin{theorem}
	Consider problem (\ref{eqn_stochastic_problem}) under Assumptions \ref{ass_boundedness_of_F}-\ref{ass_second_order_smooth}.
	If the differential estimators $\gB^{k}$ and $\HB^{k}$ satisfy Eqn. \eqref{conditiongradient} for all $k$.
	By setting $r = \sqrt{\epsilon/L_2}$ and $K = \OM(\sqrt{L_2}\Delta/\epsilon^{1.5})$, MetaAlgorithm \ref{alg_stochastic_trust_region} outputs an $\OM(\epsilon, \sqrt{\epsilon})$-SOSP in expectation.
\end{theorem}
\begin{proof}
	First of all, if our algorithm terminates when we meet $\|\hB^{k}\|<r$, from the complementary property \eqref{eqn_optimality_complementary_stochastic}, we have $\lambda^{k}=0$.
	Then we directly have $\xB^{k+1}$ is an $\OM(\epsilon, \sqrt{\epsilon})$-SOSP following the proof of Theorem \ref{thm_TR_meta}.
	
	In the following, we focus on the case when we always have $\|\hB^{k}\|=r$.
	Use such property and follow the proof of Theorem \ref{thm_TR_meta} to obtain the inequality (same as (\ref{eqn_proof_3_stochastic})):
	\begin{equation*}
	F(\xB^{k+1}) \leq F(\xB^{k}) - \frac{L_2\lambda^{k}}{4}\cdot\frac{\epsilon}{L_2} + \frac{1}{3}\cdot\frac{\epsilon^{1.5}}{\sqrt{L_2}}.
	\end{equation*}
	Summing the above inequality from $k=0$ to $K$, we have
	\begin{equation*}
		\frac{\epsilon}{K+1}\sum_{k=0}^{K}\lambda^{k}
		\leq \frac{4(F(\xB^0) - F(\xB^{K+1}))}{K+1} + \frac{4\epsilon^{1.5}}{3\sqrt{L_2}}.
	\end{equation*}
	By sampling $\bar{k}$ uniformly from $\{0,\ldots,K\}$, we obtain
	\begin{equation*}
	\epsilon\EBB[\lambda^{\bar{k}}] \leq \frac{4\Delta}{K+1} + \frac{4\epsilon^{1.5}}{3\sqrt{L_2}}.
	\end{equation*}
	Taking $K = \frac{\sqrt{L_2}\Delta}{\epsilon^{1.5}}$, we obtain $\EBB[\lambda^{\bar{k}}] \leq \frac{32}{3}\cdot\frac{\epsilon^{0.5}}{\sqrt{L_2}}.$
	
	The rest of the proof is similar to Theorem \ref{thm_TR_meta} and we have the result.
\end{proof}
\section{Additional Experimental Results}~\label{append:more_experiment}
\subsection{Descriptions of Testing Datasets}\label{appedixdata}
We briefly introduce the seven testing datasets in the manuscript. Among them, three datasets are provided in the LibSVM website\footnote[1]{https://www.csie.ntu.edu.tw/~cjlin/libsvmtools/datasets/}, including (\textsf{a09}, \textsf{ijcnn}, \textsf{codrna}, \textsf{phishing} and \textsf{w08}). The  detailed information is summarized in Table~\ref{Tabledatasets}. We can observe that these datasets are different from each other in feature dimension, training samples, etc.
\begin{figure*}[t]
	\begin{center}
		\setlength{\tabcolsep}{0.8pt} 
		\renewcommand{\arraystretch}{0.1}
		\begin{tabular}{cccc}
			\multicolumn{2}{c}{{\small{\textsf{a9a}}}}& \multicolumn{2}{c}{{\small{
						\textsf{ijcnn}}}}\\
			\includegraphics[width=0.245\linewidth]{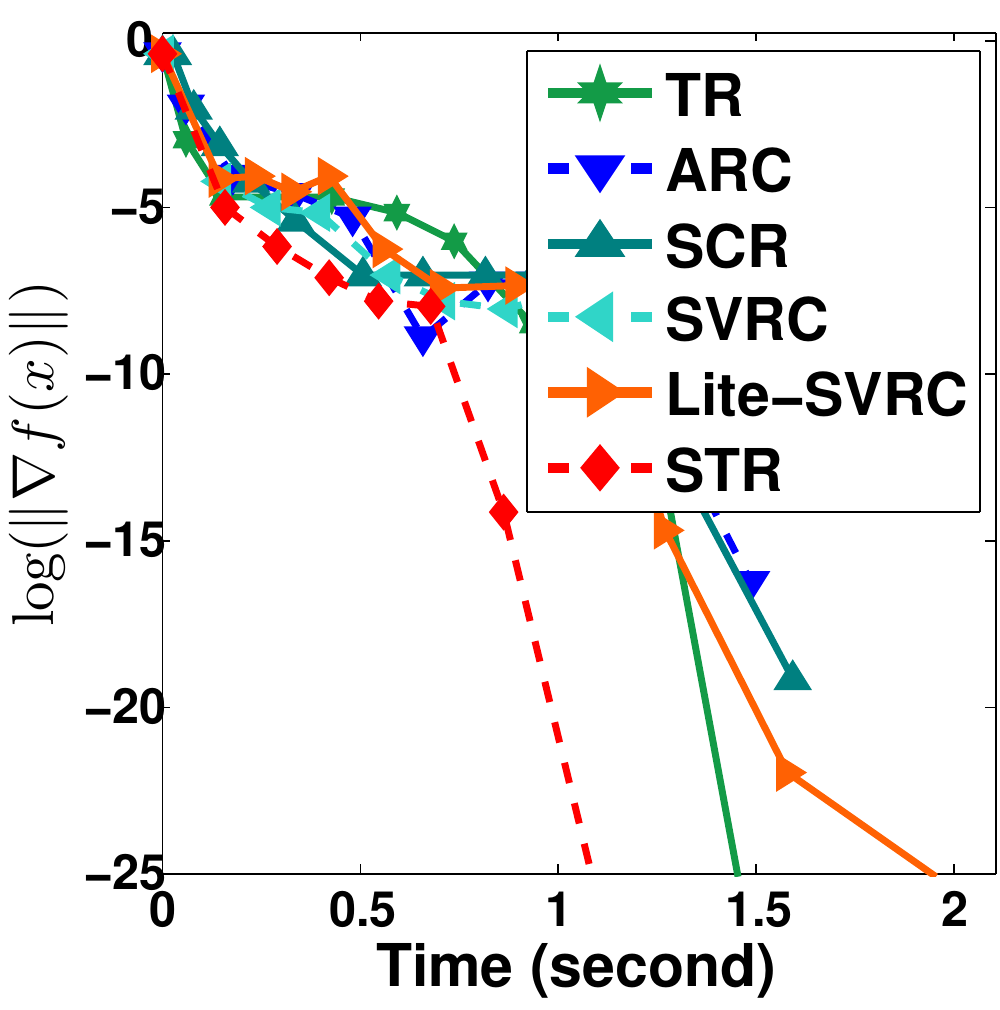}&
			\includegraphics[width=0.245\linewidth]{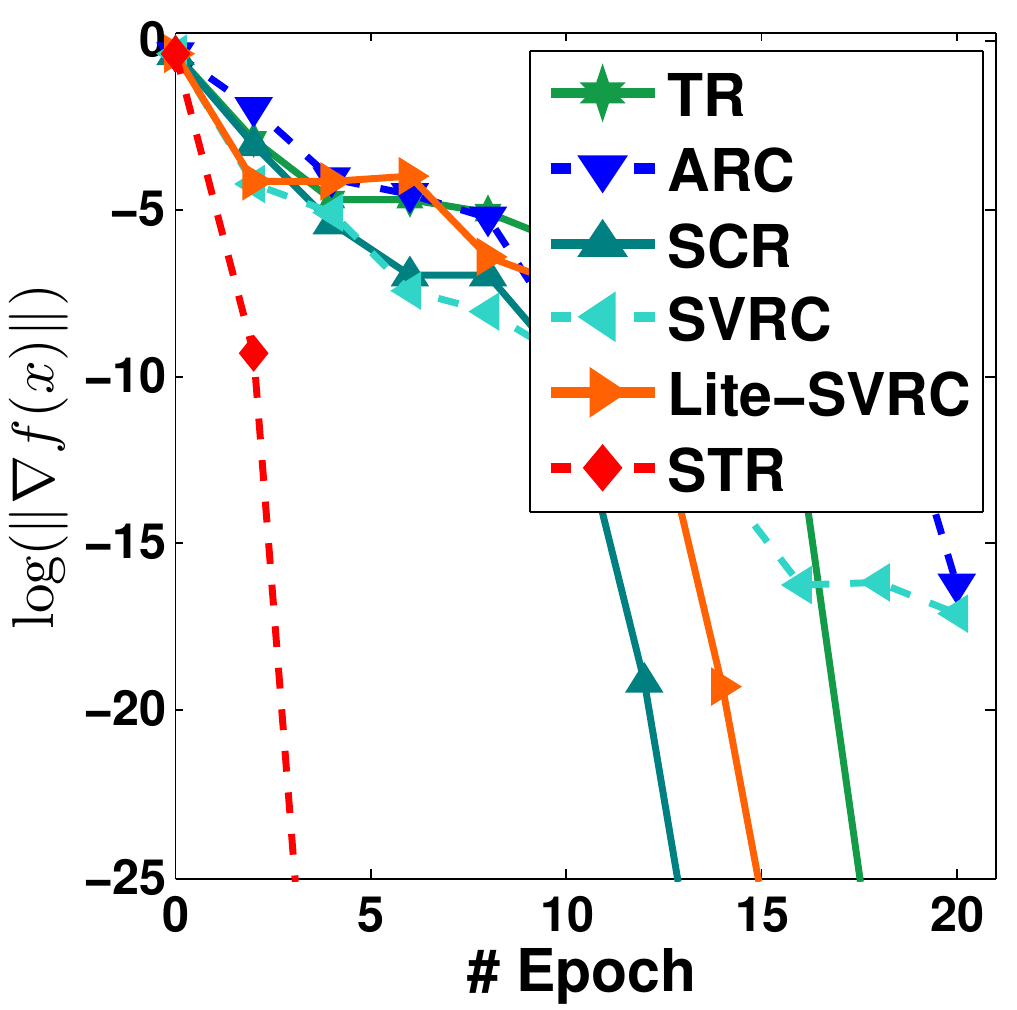}&
			\includegraphics[width=0.245\linewidth]{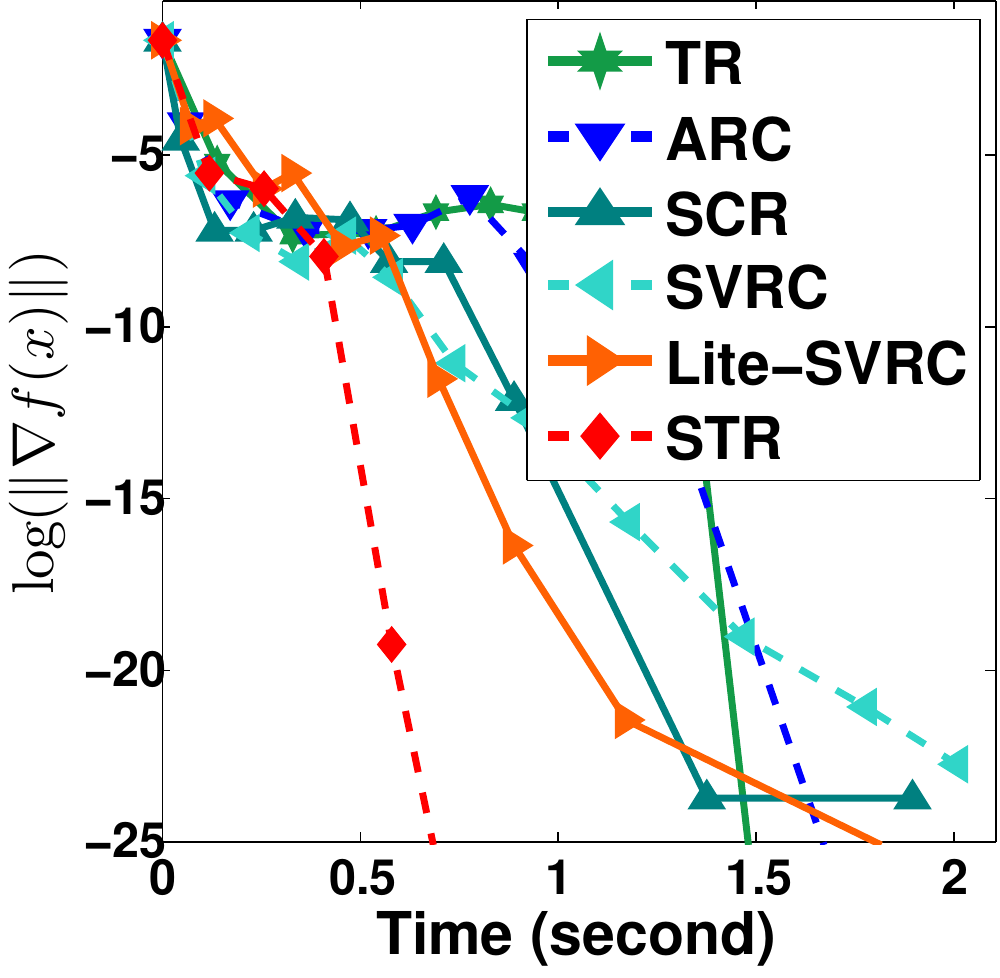}&
			\includegraphics[width=0.245\linewidth]{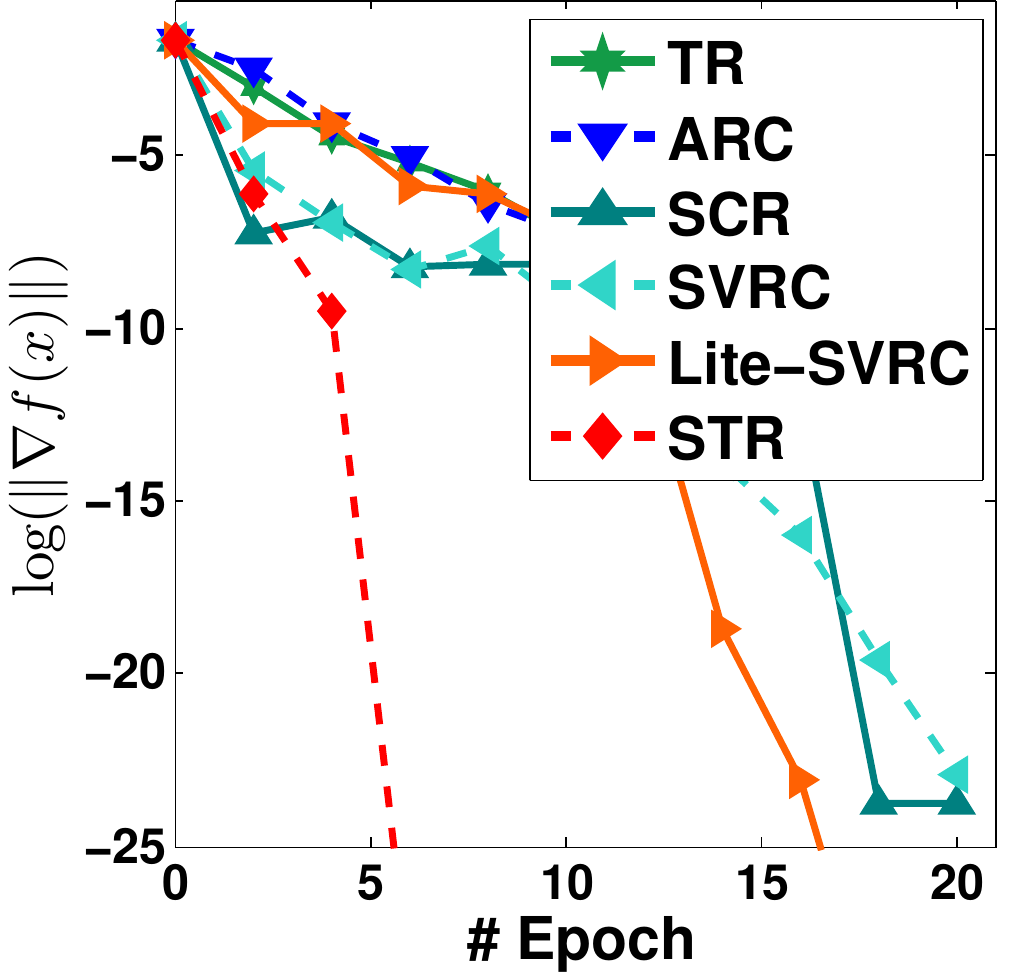}\vspace{0.5em}\\
			\multicolumn{4}{c}{{\small{(a)  nonconvex logistic regression problem}}} \\
		\end{tabular}
		\begin{tabular}{cccc}
			\multicolumn{2}{c}{{\small{\textsf{a9a}}}}& \multicolumn{2}{c}{{\small{\textsf{ijcnn}}}}\\
			\includegraphics[width=0.245\linewidth]{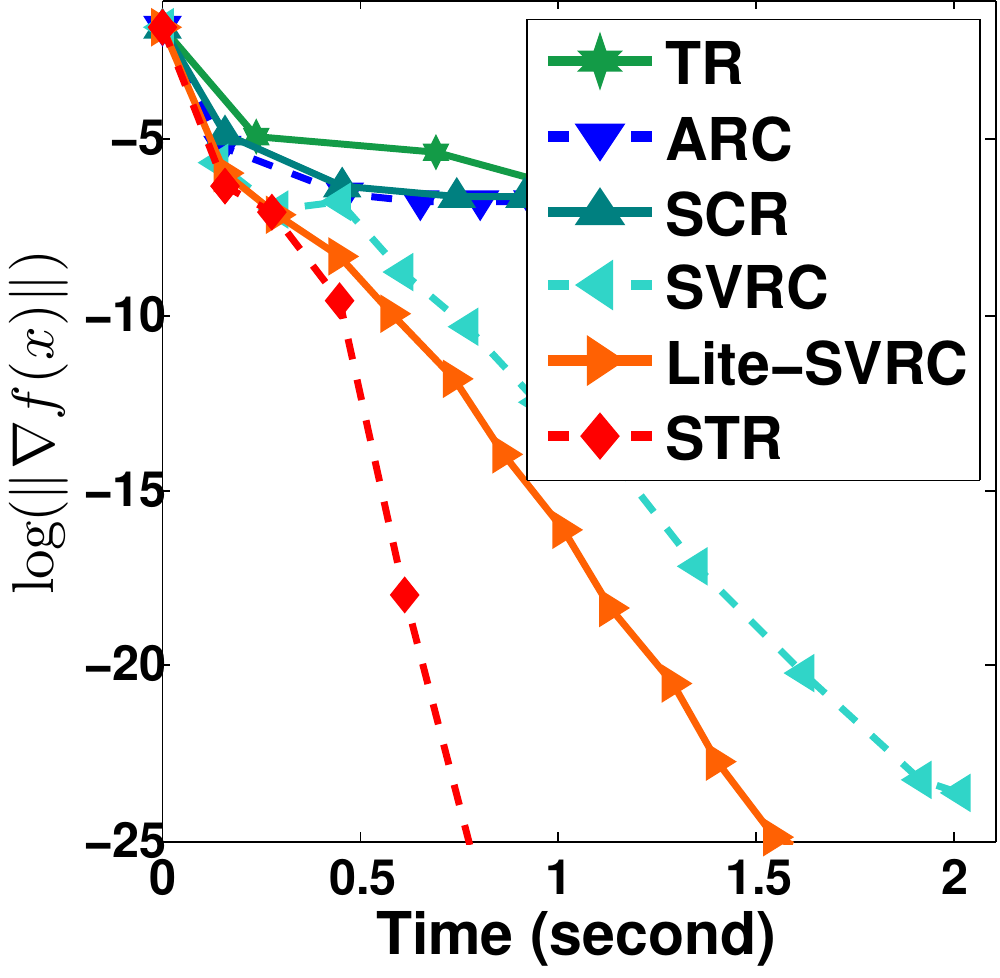}&
			\includegraphics[width=0.245\linewidth]{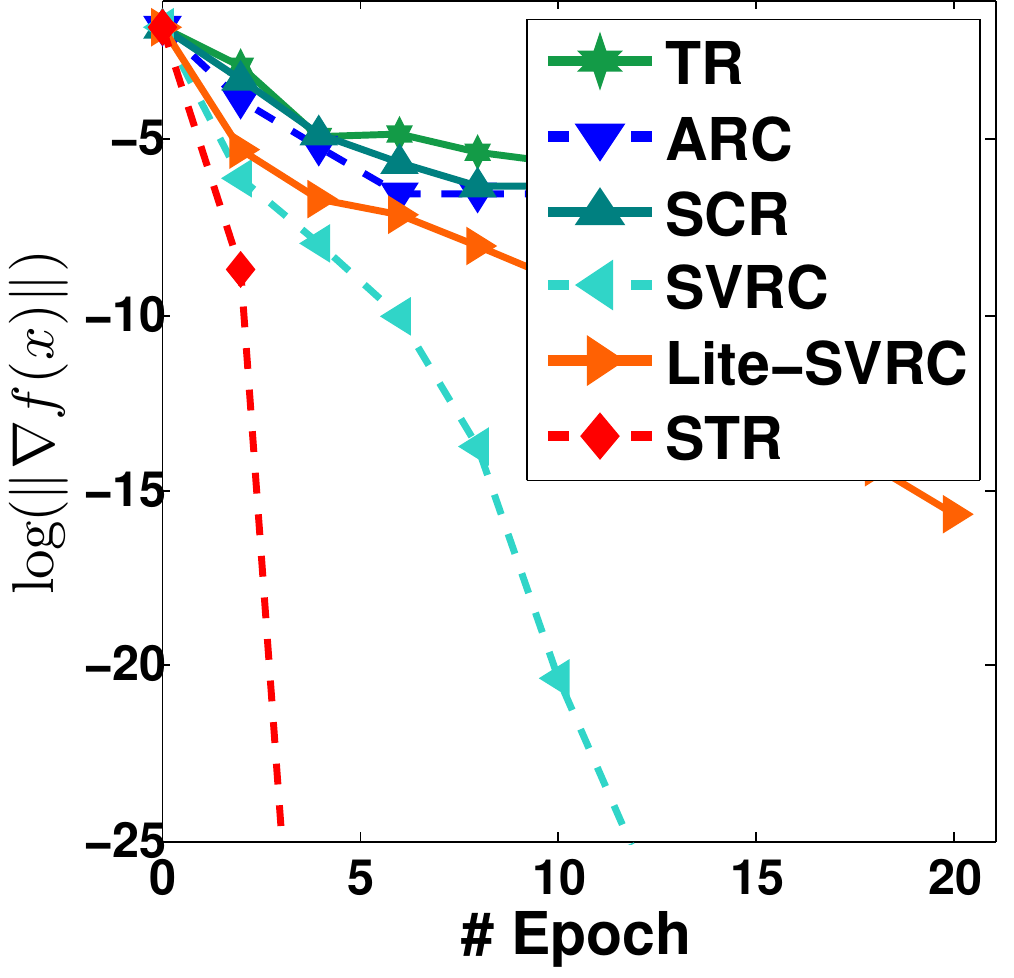}&
			\includegraphics[width=0.245\linewidth]{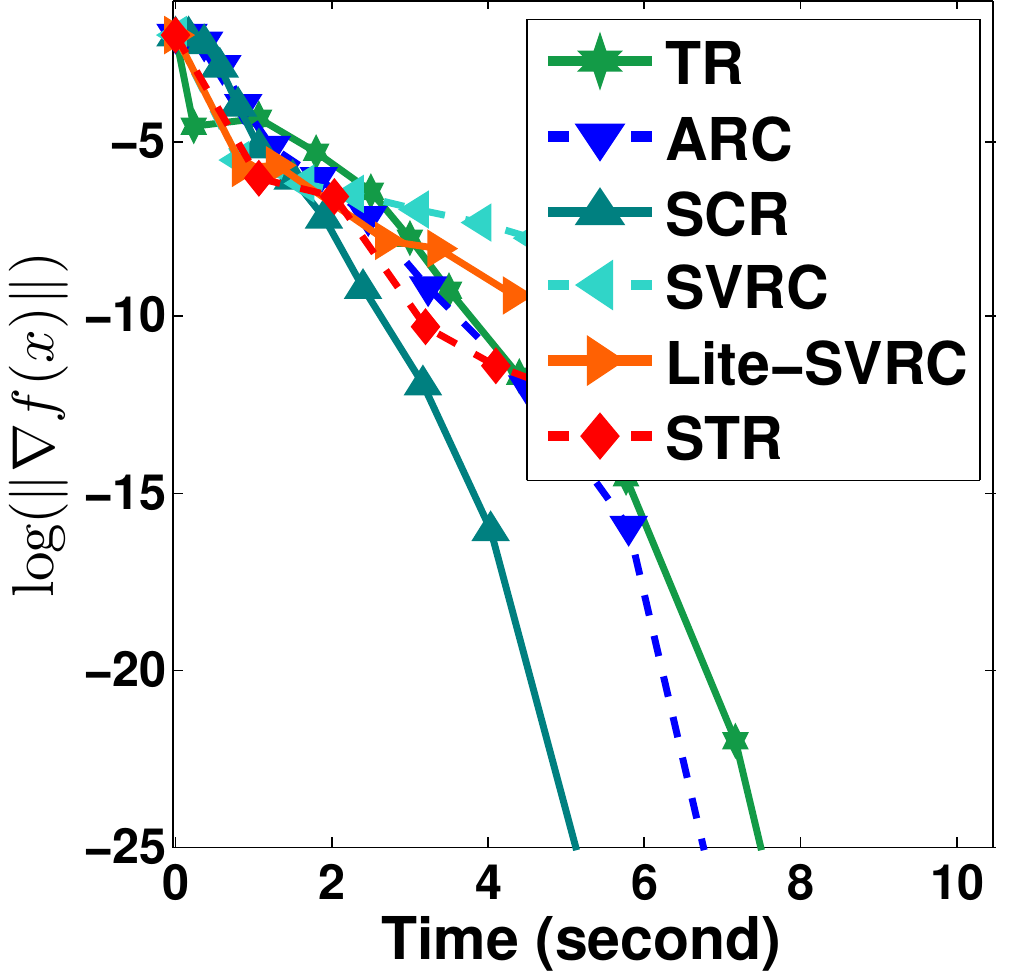}&
			\includegraphics[width=0.245\linewidth]{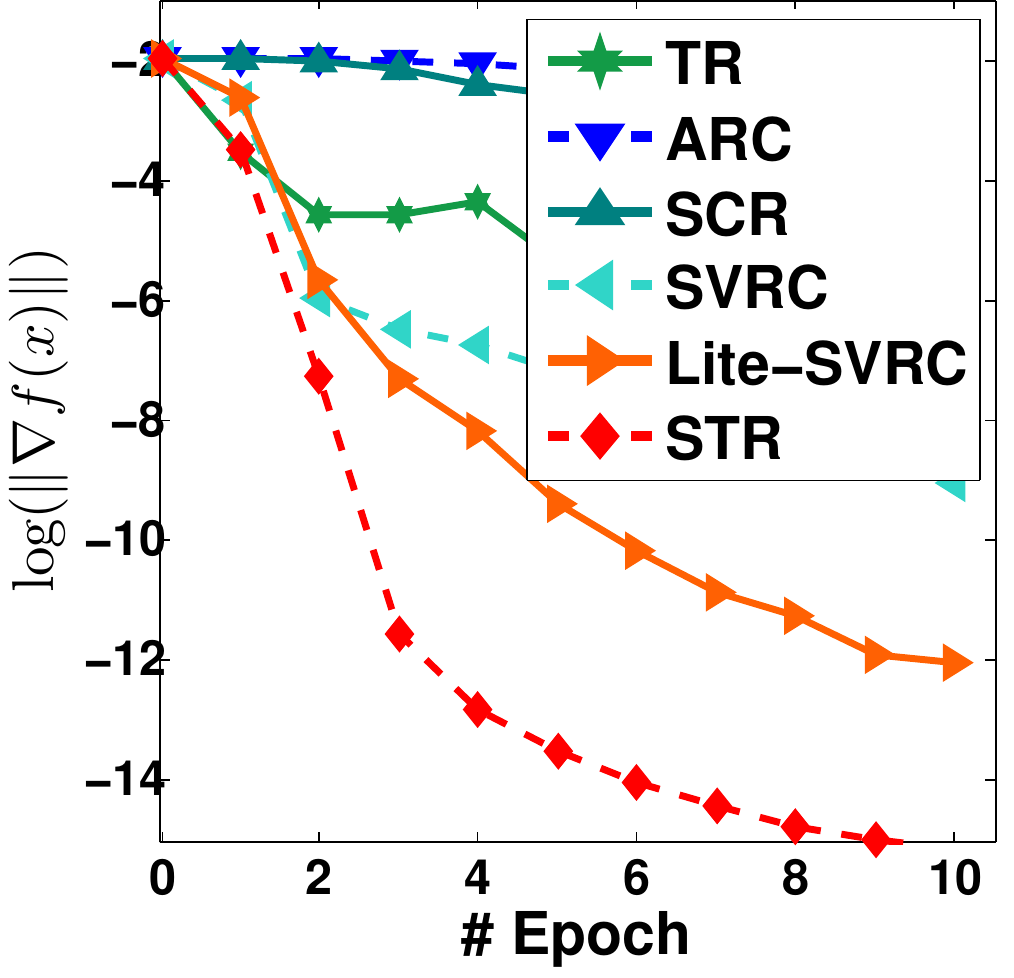}\vspace{0.5em}\\
			\multicolumn{4}{c}{{\small{(b) nonlinear least square problem}}} \\
		\end{tabular}
	\end{center}
	\caption{Comparison of gradient norm on both the nonconvex logistic regression and nonlinear least square problems. }
	\label{comparisongradient}
	\vspace{-1.0em}
\end{figure*}

\subsection{More Experiments}\label{appedxexp}
Here we give more experimental results on the gradient norm v.s. the algorithm running time and the Hessian sample complexity. Due to the space limit, in the manuscript we only provide the gradient-norm related results on the \textsf{codrna} dataset. Here we provide the results of \textsf{a9a} and \textsf{ijcnn} datasets  in Figure~\ref{comparisongradient}. One can observe that on both the logistic regression with nonconvex regularizer and the nonlinear least square problems, the proposed algorithm always shows sharper convergence behavior in terms of both the running time and the Hessian sample complexity.  These observations are consistent with the results in Figure~\ref{comparisonsquare} in the manuscript. All these results demonstrate the high efficiency of our proposed algorithm and also confirm our theoretical implication.

\end{document}